\documentclass[11pt,twopage]{amsart}
\usepackage{t1enc}
\usepackage[latin2]{inputenc}

\usepackage{amsmath}
\usepackage{amsthm}
\usepackage{amscd}
\usepackage{amssymb}
\usepackage{stmaryrd}
\usepackage[T1]{fontenc}
\usepackage{wasysym}

\usepackage[table]{xcolor}
\usepackage{booktabs}

\usepackage{verbatim}
\usepackage{amscd}

\usepackage{url}
\usepackage{color}
\usepackage[small,nohug,
]{diagrams}
\diagramstyle[labelstyle=\scriptstyle]
\newarrow{Implies} ===={=>}

\usepackage{mathpazo}

\newcommand{\0}{\emptyset}
\newcommand{\om}{\omega}

\newcommand{\mc}{\mathcal}
\newcommand{\eps}{\varepsilon}

\newcommand{\al}{\alpha}
\newcommand{\be}{\beta}
\newcommand{\ka}{\kappa}
\newcommand{\de}{\delta}
\newcommand{\lam}{\lambda}
\newcommand{\dom}{\mathrm{dom}}
\newcommand{\ran}{\mathrm{ran}}
\newcommand{\mf}{\mathfrak}
\newcommand{\mrm}{\mathrm}

\newcommand{\mbb}{\mathbb}
\newcommand{\PP}{\mathbb{P}}
\newcommand{\vd}{\Vdash}
\newcommand{\bc}{\begin{center}}
\newcommand{\ec}{\end{center}}
\newcommand{\bs}{\setminus}
\newcommand{\wh}{\widehat}
\newcommand{\add}{\mathrm{add}}
\newcommand{\cov}{\mathrm{cov}}
\newcommand{\non}{\mathrm{non}}
\newcommand{\cof}{\mathrm{cof}}
\newcommand{\clrest}{\!\upharpoonright\!}
\newcommand{\QQ}{\mathring{\mathbb{Q}}}
\newcommand{\mr}{\mathring}

\def\Ubf#1{{\baselineskip=0pt\vtop{\hbox{$#1$}\hbox{$\sim$}}}{}}

\newtheorem*{claim}{Claim}

\newtheorem{thm}{Theorem}[section]
\newtheorem{lem}[thm]{Lemma}
\newtheorem{prop}[thm]{Proposition}
\newtheorem{cor}[thm]{Corollary}
\newtheorem{fact}[thm]{Fact}

\theoremstyle{definition}

\newtheorem{prob}[thm]{Problem}
\newtheorem{df}[thm]{Definition}
\newtheorem{exa}[thm]{Example}
\newtheorem{obs}[thm]{Observation}
\newtheorem{obss}[thm]{Observations}
\newtheorem{rem}[thm]{Remark}

\title{Ways of destruction}

\thanks{The first author was supported by the Austrian Science Fund (FWF) project P29907. The second author was supported by the Austrian Science Fund (FWF) project I2374-N35.}

\author{Barnab\'as Farkas}
\address{Institut f\"ur Diskrete Mathematik und Geometrie, Technische Universit\"at Wien, Wiedner Hauptstra\ss e 8-10/104, 1040 Wien, Austria \vspace*{0.3cm}}
\email{barnabasfarkas@gmail.com}
\urladdr{http://dmg.tuwien.ac.at/farkas/}
\author{Lyubomyr Zdomskyy}
\email{lzdomsky@gmail.com}
\urladdr{http://dmg.tuwien.ac.at/zdomskyy/}


\begin{document}

\maketitle

\begin{abstract}
We study the following natural strong variant of destroying Borel ideals: $\PP$ {\em $+$-destroys} $\mc{I}$ if $\PP$ adds an $\mc{I}$-positive set which has finite intersection with every $A\in\mc{I}\cap V$. Also, we discuss the associated variants  \begin{align*}
\non^*(\mc{I},+)=&\min\big\{|\mc{Y}|:\mc{Y}\subseteq\mc{I}^+,\; \forall\;A\in\mc{I}\;\exists\;Y\in\mc{Y}\;|A\cap Y|<\om\big\}\\ \cov^*(\mc{I},+)=&\min\big\{|\mc{C}|:\mc{C}\subseteq\mc{I},\; \forall\;Y\in\mc{I}^+\;\exists\;C\in\mc{C}\;|Y\cap C|=\om\big\}
\end{align*}
of the star-uniformity and the star-covering numbers of these ideals.

Among other results, (1) we give a simple combinatorial characterisation when a real forcing $\PP_I$ can $+$-destroy a Borel ideal $\mc{J}$; (2) we discuss many classical examples of Borel ideals, their $+$-destructibility, and cardinal invariants; (3) we show that the Mathias-Prikry, $\mbb{M}(\mc{I}^*)$-generic real $+$-destroys $\mc{I}$ iff $\mbb{M}(\mc{I}^*)$ $+$-destroys $\mc{I}$ iff $\mc{I}$ can be $+$-destroyed iff $\cov^*(\mc{I},+)>\om$; (4) we characterise when the Laver-Prikry, $\mbb{L}(\mc{I}^*)$-generic real $+$-destroys $\mc{I}$, and in the case of P-ideals, when exactly $\mbb{L}(\mc{I}^*)$ $+$-destroys $\mc{I}$; (5) we briefly discuss an even stronger form of destroying ideals closely related to the additivity of the null ideal.
\end{abstract}


\section{Motivation}\label{mot}
\subsection*{Ideals on $\om$ and on Polish spaces}
If $I$ is an ideal on an infinite set $X$, we will always assume that $[X]^{<\om}=\{A\subseteq X:|A|<\om\}\subseteq I$ and $X\notin I$. Let $I^+=\mc{P}(X)\setminus I$ be the family of $I$-positive sets and $I^*=\{X\setminus A:A\in I\}$ be the dual filter of $I$. We will work with ideals on countable underlying sets, e.g.
\begin{align*}
\mc{I}_{1/n} & =\Big\{A\subseteq\om\setminus\{0\}:\sum_{n\in A}\frac{1}{n}<\infty\Big\},\\
 \mrm{Nwd} &=\big\{A\subseteq \mbb{Q}:A\;\text{is nowhere dense (in $\mbb{Q}$)}\big\},\;\text{or}\\
\mrm{Fin}\otimes\mrm{Fin} & =\big\{A\subseteq\om\times\om:\forall^\infty\;n\;\,|(A)_n|<\om\big\}
\end{align*}
where $\forall^\infty$  stands for ``for all but finitely many'', $ \exists^\infty$ for $\neg\forall^\infty\neg$, that is, for ``there is infinitely many'', and $(A)_n$ denotes $\{k:(n,k)\in A\}$; also we will work with ($\sigma$-)ideals on uncountable Polish spaces, e.g.
\begin{align*}
\mc{M} &=\big\{B\subseteq\,\!^\om 2:B\;\text{is meager (i.e. of first-category)}\big\},\\
\mc{N} &=\big\{B\subseteq\,\!^\om 2:B\;\text{is of Lebesgue-measure null}\big\},\;\text{or}\\
\mc{K}_\sigma & =\big\{B\subseteq\,\!^\om\om:B\;\text{can be covered by a $\sigma$-compact set}\big\}
\end{align*}
where $\,\!^\om 2$ and $\,\!^\om\om$ are equipped with the usual Polish product topologies, that is, these topologies are generated by the clopen sets $\{f\in\,\!^\om\ell:t\subseteq f\}$ where $\ell=2$ or $\ell=\om$ and $t\in\,\!^{<\om}\ell=\{s:s$ is a function, $\dom(s)\in\om$, and $\mrm{ran}(s)\subseteq\ell\}$. The Cantor space $\,\!^\om 2$ is compact and the measure we referred to as the Lebesgue-measure above is the  product probability measure (the power of the uniform distribution on $2$).

By identifying $\mc{P}(\om)$ and $\,\!^\om 2$, we can talk about measure, category, and complexity of subsets of $\mc{P}(\om)$, in particular, of ideals on $\om$, and similarly on arbitrary countably infinite underlying sets (e.g. $\mc{I}_{1/n}$ is $F_\sigma$, $\mrm{Nwd}$ is $F_{\sigma\delta}$, and $\mrm{Fin}\otimes\mrm{Fin}$ is $F_{\sigma\delta\sigma}$). In Section \ref{prel}, we will present many classical examples of ideals on countable underlying sets.

\smallskip
Concerning combinatorial properties and cardinal invariants of definable (typically Borel) ideals in forcing extensions, one of the most crucial points is to understand whether a forcing notion destroys an ideal, and if so, how ``strongly''. We are interested in various notions of destroying ideals, in their possible characterisations, in their interactions with classical properties of forcing notions, and in the associated cardinal invariants.

We will mainly focus on classical forcing notions, and in general on forcing notions (can be written / equivalent to one) of the form $\PP_I=(\mc{B}(X)\setminus I,\subseteq)$ where $X$ is an uncountable Polish space, $\mc{B}(X)=\{$Borel subsets of $X\}$, and $I$ is a $\sigma$-ideal on $X$ $\sigma$-generated by a ``definable'' family of Borel sets (see later). For example, $\mbb{C}=\PP_\mc{M}$ is the Cohen forcing, $\mbb{B}=\PP_\mc{N}$ is the random forcing, and $\PP_{\mc{K}_\sigma}$ is the Miller forcing. In general, we know (see \cite[Prop. 2.1.2]{zap}) that $\PP_I$ adds a ``real'' $r_I\in X$ ($X$ can be seen as a $G_\delta$ subset of $^\om\,\![0,1]$) determined by the following property: If $V$ is a transitive model (of a large enough finite fragment) of $\mrm{ZFC}$, $I$ (more precisely, the family $\sigma$-generating $I$) is coded in $V$, $G$ is $\PP_I$-generic over $V$, and $B\in\PP_I$ is a Borel set coded in $V$, then $r_I\in B$ iff $B\in G$\footnote{When working with $\PP_I$, sometimes we refer to $\PP_I$ in the universe and sometimes to its interpretation in a transitive model but this should always be clear from the context. Similarly when working with Borel sets, for example in models or in a formula of the forcing language, we refer to their definition, e.g. if $B$ is coded in $V$, then $B\in G$ means that the interpretation $B^V=B\cap V$ of $B$'s code in $V$ belongs to $G$; and $p\vd_\PP \mr{x}\in B$ means that $\mr{x}[G]\in B^{V[G]}$ (or simply $\mr{x}[G]\in B$) for every $\PP$-generic $G$ (over $V$) containing $p$.} (see \cite{zap} for a detailed study of these forcing notions).

\subsection*{Destroying ideals} Let us recall the classical notion of forcing (in)destructibility: We say that an ideal $\mc{I}$ on $\om$ is {\em tall} if every infinite $X\subseteq\om$ contains an infinite element of $\mc{I}$, e.g. $\mc{I}_{1/n}$, $\mrm{Nwd}$, and $\mrm{Fin}\otimes\mrm{Fin}$ are tall. A forcing notion $\PP$ {\em can destroy} $\mc{I}$ if there is a condition $p\in\PP$ such that
\begin{align*} & p\vd\text{``the ideal generated by $\mc{I}^V$ is not tall'' i.e.}\\
& p\vd\exists\;Y\in [\om]^\om\;\forall\;A\in\mc{I}^V\;|Y\cap A|<\om
\end{align*}
where we write $\mc{I}^V$ to make it completely clear that even in the case of definable ideals, we refer to the ideal from (or interpreted in) the ground model.\footnote{Of course, we could also write $\check{\mc{I}}$ here, referring to the canonical $\PP$-name of $\mc{I}$, but as usual, in the forcing language we will not use any specific notions for ground model objects.}  We say that $\PP$ {\em destroys} $\mc{I}$ if $p=1_\PP$, and that $\mc{I}$ is {\em $\PP$-indestructible} if $\PP$ cannot destroy $\mc{I}$.

\smallskip
We know that every ideal can be destroyed by a $\sigma$-centered forcing notion: Let $\mc{I}$ be arbitrary and define the associated {\em Mathias-Prikry forcing} $\mbb{M}(\mc{I}^*)$ as follows (see \cite{Canjar}, \cite{CRZ}, and \cite{HrMi}): $(s,F)\in\mbb{M}(\mc{I}^*)$ if $s\in [\om]^{<\om}$ and $F\in\mc{I}^*$; $(s_0,F_0)\leq (s_1,F_1)$ if $s_0$ end-extends $s_1$ (with respect to a fixed enumeration of the underlying set of $\mc{I}$), $F_0\subseteq F_1$, and $s_0\setminus s_1\subseteq F_1$. We know that $\mbb{M}(\mc{I}^*)$ is $\sigma$-centered (conditions with the same first coordinates are compatible), and it destroys $\mc{I}$: If $G$ is $(V,\mbb{M}(\mc{I}^*))$-generic and $Y_G=\bigcup\{s:(s,F)\in G$ for some $F\}$, then $Y_G\in [\om]^\om\cap V[G]$ and $|Y_G\cap A|<\om$ for every $A\in\mc{I}^V$.

\smallskip
Sometimes $\mbb{M}(\mc{I}^*)$ does more than just ``simply'' destroying $\mc{I}$: Trivial density arguments show that if $\mc{I}=\mc{I}_{1/n}$ or $\mc{I}=\mrm{Nwd}$ then $V^{\mbb{M}(\mc{I}^*)}\models Y_{\mr{G}}\in\mc{I}^+$ (where $\mc{I}^+$ is defined in the extension of course). In general, $Y_{\mr{G}}$ is not necessarily $\mc{I}$-positive: If $\mc{I}=\mrm{Fin}\otimes\mrm{Fin}$ and $Y\in \mc{I}^+\cap V^\PP$, then $Y\cap (\{n\}\times\om)$ is infinite for infinitely many $n$ and $\{n\}\times\om\in \mc{I}^V$ for every $n$, in other words, no forcing notion can add a $\mc{I}$-positive set which is almost disjoint from all elements of $\mc{I}^V$. In the case of this stronger notion of destruction we need definability, we will focus on Borel, sometimes analytic or coanalytic ideals. If $\mc{I}$ is analytic or coanalytic and $\PP$ adds a $\mr{Y}\in\mc{I}^+$ such that $|\mr{Y}\cap A|<\om$ for every $A\in\mc{I}^V$, then we will say that $\PP$ {\em $+$-destroys } (or {\em can $+$-destroy}) $\mc{I}$. We will show (see Corollary \ref{mathias}) that if a Borel ideal $\mc{I}$ can be $+$-destroyed, then $\mbb{M}(\mc{I}^*)$ $+$-destroys it.

Why do we prefer (at most) analytic or coanalytic ideals? If $\mc{I}$ is $\Ubf{\Sigma}^1_1$ or $\Ubf{\Pi}^1_1$, then, applying the Mostowski Absoluteness Theorem, $\mc{I}^V=\mc{I}\cap V$ whenever $V$ is a transitive model of $\mrm{ZFC}$ and $\mc{I}$ is coded $V$. Furthermore, if $\mc{X}\subseteq\mc{P}(\om)$ is an analytic or coanalytic set, then ``$\mc{X}$ is an ideal'' is a $\Ubf{\Pi}^1_2$ statement and hence, by the Shoenfield Absoluteness Theorem, it is absolute between $V$ and $V^\PP$ assuming $\mc{X}$ is coded in $V$  (in general, between transitive models $V\subseteq W$ satisfying $\om_1^W\subseteq V$).

\smallskip
One may ask now if we can go even further and add a set $Z\in\mc{I}^*\cap V^\PP$ which has finite intersection with every $A\in\mc{I}\cap V$ (where $\mc{I}$ is analytic or coanalytic), if so, we say that $\PP$ {\em $*$-destroys} (or {\em can $*$-destroy}) $\mc{I}$. Let us point out certain crucial observations concerning $*$-destructibility:

1) If we can add such a $Z$, then $A\subseteq^* \om\setminus Z\in\mc{I}$ (where $X\subseteq^* Y$ iff $|X\setminus Y|<\om$) for every $A\in \mc{I}^V$. Therefore $\mc{I}$ must be a {\em P-ideal}, that is, for every countable $\mc{A}\subseteq\mc{I}$ there is a {\em pseudounion} $B\in\mc{I}$ of $\mc{A}$, that is, $A\subseteq^* B$ for every $A\in\mc{A}$ (e.g. $\mc{I}_{1/n}$ is a P-ideal but $\mrm{Nwd}$ and $\mrm{Fin}\otimes\mrm{Fin}$ are not). Why? The formula ``$x=(x_n)_{n\in\om}\in\,\!^\om\mc{P}(\om)$ is a sequence in $\mc{I}$ without pseudounion in $\mc{I}$'' is $\Ubf{\Pi}^1_2$ hence absolute between $V$ and $V^\PP$.

2) A $\sigma$-centered forcing notion cannot $*$-destroy any tall analytic P-ideal $\mc{I}$
(see \cite[Thm. 6.4]{FaSo}), in particular, $\mbb{M}(\mc{I}^*)$ cannot $*$-destroy $\mc{I}$. The same holds for {\em somewhere tall} analytic P-ideals, that is, when $\mc{I}\clrest X=\{A\subseteq X:A\in\mc{I}\}$ is tall for some $X\in\mc{I}^+$. What can we say about nowhere tall analytic P-ideals? Applying Solecki's characterisation of analytic P-ideals, one can show (see later) that up to isomorphism (via a bijection between the underlying sets, in notation $\mc{I}\simeq\mc{J}$) there are only three nowhere tall analytic P-ideals: {\em trivial modifications of} $\mrm{Fin}=[\om]^{<\om}$, that is, ideals of the form $\{A\subseteq\om:|A\cap X|<\om\}$ for an infinite $X\subseteq\om$ (clearly, there are two nonisomorphic ideals of this form); and the density ideal (see below for the definition of density ideals) \[ \{\0\}\otimes\mrm{Fin}=\big\{A\subseteq \om\times\om:\forall\;n\;|(A)_n|<\om\big\}.\]
Clearly, every forcing notion $*$-destroys trivial modifications of $\mrm{Fin}$, and $\PP$ $*$-destroys $\{\0\}\otimes\mrm{Fin}$ iff $\PP$ adds a dominating real (and hence in these three special cases, $*$-destruction is possible by $\sigma$-centered forcing notions).

3) And finally, we know that every analytic P-ideal $\mc{I}$ can be $*$-destroyed: We either use Solecki's characterisation and an ad hoc construction for a fixed analytic P-ideal, or consider the localization forcing (see below or \cite[Lem. 3.1]{towers}).

\subsection*{The role of the Kat\v{e}tov(-Blass) preorder}
Probably the most well-known characterisation of (classical) forcing destructibility of ideals is via Kat\v{e}tov-reductions to trace ideals (see \cite{BrYa} and \cite{quot}). If $\mc{I}$ and $\mc{J}$ are ideals on $\om$, then $\mc{I}$ is {\em Kat\v{e}tov-below} $\mc{J}$,
\[ \mc{I}\leq_\mrm{K}\mc{J}\;\;\text{iff}\;\;\exists\;f\in\,\!^\om\om\;\forall\;A\subseteq\om\;\big(A\in\mc{I}\longrightarrow f^{-1}[A]\in\mc{J}\big).\]
If we restrict $f$ to be finite-to-one in this definition, we obtain the {\em Kat\v{e}tov-Blass-preorder}, $\leq_\mrm{KB}$. These preorders play a fundamental role in characterising combinatorial properties of ideals (see e.g. \cite{hrusaksummary}, \cite{hrusakkatetov}, \cite{weakq}, \cite{ramseyprop}, \cite{solid}, \cite{kanoree}, \cite{hausd}, \cite{ppoints}). Let us point out here that if $\mc{I}$ and $\mc{J}$ are Borel ideals, then the statement ``$\mc{I}\leq_\mrm{K(B)}\mc{J}$'' is $\Ubf{\Sigma}^1_2$ and hence absolute between $V$ and $V^\PP$.

For $A\subseteq\,\!^{<\om}\ell$ where $\ell=2$ or $\ell=\om$, define the {\em $G_\delta$-closure} of $A$ as
\[ [A]_\delta=\big\{x\in \,\!^\om\ell:\exists^\infty\;n\;x\clrest n\in A\big\};\footnote{Notice that $[A]_\delta$ does not depend on $\ell$, even if we allow $\ell$ to be any countable set, because $[A]_\delta=\{x:x$ is a function, $\dom(x)=\om$, and $\exists^\infty$ $n$ $x\clrest n\in A\}$.}\] and for any ideal $I$ on $\,\!^\om \ell$ define the {\em trace of $I$}, an ideal on $\,\!^{<\om}\ell$, as follows \[ \mrm{tr}(I)=\big\{A\subseteq\,\!^{<\om}\ell:[A]_\delta\in I\big\}.\]
For example if $\mrm{NWD}$ is the ideal of nowhere dense subsets of $\,\!^\om 2$, then $\mrm{tr}(\mrm{NWD})=\mrm{tr}(\mc{M})=$
\[ \big\{A\subseteq \,\!^{<\om} 2:\forall\;s\in \,\!^{<\om} 2\;\exists\;t\in \,\!^{<\om} 2\;\big(s\subseteq t\;\text{and}\;A\cap t^\uparrow=\0\big)\big\}\simeq\mrm{Nwd}\]
where $t^\uparrow=\{t'\in \,\!^{<\om} 2:t\subseteq t'\}$. It is trivial to see that $\PP_I$ destroys $\mrm{tr}(I)$: If $\mr{r}_I\in \,\!^\om \ell\cap V^{\PP_I}$ is the $\PP_I$-generic real over $V$, $\mr{R}=\{\mr{r}_I\clrest n:n\in\om\}\in [\,\!^{<\om}\ell]^\om\cap V^{\PP_I}$, $B\in\PP_I$, and $A\in\mrm{tr}(I)$, then $B'=B\setminus [A]_\delta\in\PP_I$, $B'\leq B$, and $B'\vd |A\cap \mr{R}|<\om$.

\begin{thm}\label{origchar} {\em (see \cite[Thm. 1.6]{quot})}
Let $I$ be a $\sigma$-ideal on $\,\!^\om\ell$ such that $\PP_I$ is proper and $I$ satisfies the continuous readings of names (CRN, see below), and let $\mc{J}$ be an ideal on $\om$. Then $\PP_I$ can destroy $\mc{J}$ if, and only if $\mc{J}\leq_\mrm{K}\mrm{tr}(I)\clrest X$ for some $X\in\mrm{tr}(I)^+$.
\end{thm}

The paper is organised as follows: In Section \ref{prel}, we present many classical Borel ideals, as well as the characterisations of $F_\sigma$ ideals and of analytic P-ideals (due to Mazur and Solecki). In Section \ref{degrees}, we give a detailed introduction to the notions of destructibility of ideals and to the associated cardinal invariants, also, we present a combinatorial characterisation of forcing (in)destructibility by proper real forcing notions. In Section \ref{secexa}, we discuss our examples of non P-ideals from Section \ref{prel} in the context of $+$-destructibility and the new cardinal invariants. In Section \ref{secmatlav}, applying Laflamme's filter games and his results, we characterise when the Mathias-Prikry and Laver-Prikry generic reals, and in the case of the first one, the forcing notion in general, $+$-destroy the defining ideal. In Section \ref{secfrag}, we characterise when exactly the Laver-Prikry forcing $+$-destroys the defining P-ideal. In Section \ref{secstardest}, we present a survey on $*$-destructibility and its connection to the null ideal. Finally, in Section \ref{secque}, we list all our remaining open questions.

\section{Borel ideals}\label{prel}
We present some additional classical examples of Borel ideals (for their specific roles in characterisation results see e.g. \cite{hrusaksummary} or  \cite{hrusakkatetov}, also find more citations below). We already defined $\mc{I}_{1/n}$, $\mrm{Nwd}$, $\mrm{Fin}\otimes\mrm{Fin}$, and $\{\0\}\otimes\mrm{Fin}$.

\smallskip
{\em Summable ideals} (a.k.a. generalisations of $\mc{I}_{1/n}$): Let $h:\om\to [0,\infty)$ such that $\sum_{n\in\om}
h(n)=\infty$. Then the {\em summable ideal generated by $h$} is
\[ \mc{I}_h=\bigg\{A\subseteq\om:\sum_{n\in A} h(n)<\infty\bigg\}.\]
$\mc{I}_h$ is an $F_\sigma$ P-ideal, and it is tall iff
$\lim_{n\to\infty}h(n)=0$.

\smallskip
{\em Eventually different ideals}: Let
\[ \mc{ED}=\Big\{A\subseteq\omega\times\om:\limsup_{n\in\om} |(A)_n|<\infty\Big\},\]
$\Delta=\{(n,k)\in\om\times\om:k\leq n\}$, and $\mc{ED}_\mrm{fin}=\mc{ED}\clrest \Delta$. Then $\mc{ED}$ and $\mc{ED}_\mrm{fin}$ are tall $F_\sigma$ non P-ideals.

\smallskip
The {\em random graph ideal}: Let
\[ \mrm{Ran}=\mrm{id}\big(\big\{\text{homogeneous subsets of the random graph}\big\}\big)\]
where the {\em random graph} $(\om,E)$, $E\subseteq [\om]^2$ is up to isomorphism uniquely determined by the following property: For every pair $A,B\subseteq\om$ of nonempty, finite, disjoint sets, there is an $n\in\om\setminus (A\cup B)$ such that $\{\{n,a\}:a\in A\}\subseteq E$ and $\{\{n,b\}:b\in B\}\cap E=\0$. A set $H\subseteq\om$ is ($E$-){\em homogeneous} iff $[H]^2\subseteq E$ or $[H]^2\cap E=\0$; and $\mrm{id}(\mc{H})$ stands for the ideal generated by $\mc{H}$ (that is, the collection of all subsets of $\bigcup\mc{H}$ which can be covered by finitely many elements of $\mc{H}$, of course in general $\mrm{id}(\mc{H})$ is not necessarily proper). $\mrm{Ran}$ is a tall $F_\sigma$ non P-ideal.

\smallskip
{\em Solecki's ideal} (see \cite{solid}, \cite{kanoree}, and \cite{hausd}): Let $\mrm{CO}(\,\!^\om 2)$ be the family of clopen subsets of $\,\!^\om 2$ and $\Omega=\{C\in\mrm{CO}(\,\!^\om 2):\lam(C)=1/2\}$ where $\lam$ is the Lebesgue-measure on $\,\!^\om 2$ (clearly, $C$ is clopen iff $C$ is a union of finitely many basic clopen sets, and hence $|\mrm{CO}(\,\!^\om 2)|=|\Omega|=\om$). The ideal $\mc{S}$ on $\Omega$ is generated by $\{\mc{C}_x:x\in\,\!^\om 2\}$ where $\mc{C}_x=\{C\in\Omega:x\in C\}$. $\mc{S}$ is a tall $F_\sigma$ non P-ideal.

\smallskip
{\em Density and generalised density ideals.}
Let $(P_n)_{n\in\omega}$ be a partition of $\om$ into nonempty finite sets and let $\vec{\vartheta}=(\vartheta_n)_{n\in\om}$ be a
sequences of measures or submeasures (in the generalised case, see the definition below), $\vartheta_n:\mc{P}(P_n)\to [0,\infty)$ such that $\limsup_{n\to\infty}\vartheta_n(P_n)>0$. The {\em (generalised) density ideal generated by $\vec{\vartheta}$}
is
\[ \mc{Z}_{\vec{\vartheta}}=\Big\{A\subseteq\om:\lim_{n\to\infty}\vartheta_n(A\cap P_n)=0\Big\}.\]
Ideals of this form are $F_{\sigma\delta}$ P-ideals, and the ideal $\mc{Z}_{\vec{\vartheta}}$ is tall iff $\max\{\vartheta_n(\{k\}):k\in P_n\}\xrightarrow{n\to\infty}0$. The {\em density zero ideal}
\[ \mc{Z}=\big\{A\subseteq\om\setminus\{0\}:|A\cap n|/n\to 0\big\}=\bigg\{A\subseteq\om\setminus\{0\}:\frac{|A\cap [2^n,2^{n+1})|}{2^n}\to 0\bigg\}\] is a tall density ideal. It is easy to see that $\mc{I}_{1/n}\subsetneq\mc{Z}$. Also, it is straightforward to check that $\{\0\}\otimes\mrm{Fin}$ is a density ideal.

\smallskip
The {\em trace ideal of the null ideal}:  \[\mrm{tr}(\mc{N})=\big\{A\subseteq \,\!^{<\om} 2:[A]_\delta\in \mc{N}\big\}\]
is a tall $F_{\sigma\delta}$ P-ideal (but in general, trace ideals can be very complex, see \cite[Prop. 5.1]{quot}).

\smallskip
The ideal $\mrm{Conv}$ is generated by those infinite subsets of $\mbb{Q}\cap [0,1]$ which are convergent in $[0,1]$, in other words
\[ \mrm{Conv}=\big\{A\subseteq \mbb{Q}\cap [0,1]:|\text{accumulation points of $A$ (in}\;\mbb{R})|<\om\big\}.\]
This ideal is a tall, $F_{\sigma\delta\sigma}$, non P-ideal.

\smallskip
It is easy to see that there are no $G_\delta$ (i.e. $\Ubf{\Pi}^0_2$) ideals, and we know that there are many $F_\sigma$ (i.e. $\Ubf{\Sigma}^0_2$) ideals. In general, we know  (see \cite{Cal85} and \cite{Cal88}) that there are $\Ubf{\Sigma}^0_\al$- and $\Ubf{\Pi}^0_\al$-complete ideals for every $\al\geq 3$. About ideals on the ambiguous levels of the Borel hierarchy see \cite{Eng}. For projective examples, see \cite{adrmix}.

\smallskip
Kat\v{e}tov and Kat\v{e}tov-Blass reducibilities between our main examples have been extensively studied (see e.g. \cite{hrusakkatetov}, \cite{janajorg}, and \cite{towers}), and apart from the very few unknown reducibilities (e.g. the still open $\mrm{Ran}\leq_\mrm{K(B)}\mc{S}$), we are provided with a quite satisfying ``map'' of Kat\v{e}tov-reducibilities between our main examples. Moreover, all these reductions can be chosen as finite-to-one (i.e. as KB-reductions), and in almost all cases, if we know that there is a no KB-reduction then there is no K-reduction either between these examples.

\subsection*{$F_\sigma$ ideals and analytic P-ideals} There is a natural way of defining nice ideals on $\om$ from submeasures.
A function $\varphi:\mc{P}({\omega})\to [0,\infty]$
is a {\em submeasure on $\om$} if $\varphi(\0)=0$; ${\varphi}(X)\le \varphi(X\cup Y)\le \varphi(X)+\varphi(Y)$ for every $X,Y\subseteq\om$; and $\varphi(\{n\})<\infty$ for every $n\in {\omega}$. $\varphi$ is {\em lower semicontinuous} (lsc, for short) if $\varphi(X)=\sup\{\varphi(X\cap n):n\in\om\}$ for each $X\subseteq\om$.

If $\varphi$ is an lsc submeasure on $\om$ then for $X\subseteq\om$ let $\|X\|_\varphi=\lim_{n\to\infty}\varphi(X\bs n)$.
We assign two ideals to a submeasure $\varphi$ as follows
\begin{align*}
\mrm{Fin}(\varphi) & =  \big\{X\subseteq\om:\varphi(X)<\infty\big\},\\
\mrm{Exh}(\varphi) & =  \big\{X\subseteq {\omega}:\|X\|_\varphi=0\big\}.
\end{align*}
It is easy to see that if $\mrm{Fin}(\varphi)\ne\mc{P}(\om)$, then it is an $F_\sigma$ ideal; and similarly if $\mrm{Exh}(\varphi)\ne\mc{P}(\om)$, then it is an $F_{\sigma\delta}$ P-ideal. Clearly, $\mc{I}_{\varphi(\{\cdot\})}\subseteq\mrm{Exh}(\varphi)\subseteq\mrm{Fin}(\varphi)$ always holds where $\mc{I}_{\varphi(\{\cdot\})}$ stands for the summable ideal generated by the sequence $(\varphi(\{n\}))_{n\in\om}$. From now on, when working with $\mrm{Fin}(\varphi)$ or $\mrm{Exh}(\varphi)$, we will always assume that they are proper ideals.
It is straightforward to see that if $\varphi$ is an lsc submeasure on $\om$ then $\mrm{Exh}(\varphi)$ is tall iff $\lim_{n\to\infty}\varphi(\{n\})=0$.

\begin{exa}
If $\mc{I}_h$ is a summable ideal then $\mc{I}_h=\mrm{Fin}(\varphi_h)=\mrm{Exh}(\varphi_h)$ where $\varphi_h(A)=\sum_{n\in A}h(n)$; if $\mc{Z}_{\vec\vartheta}$ is a generalised density ideal, then $\mc{Z}_{\vec\vartheta}=\mrm{Exh}(\varphi_{\vec\vartheta})$ where $\varphi_{\vec\vartheta}(A)=\sup\{\vartheta_n(A\cap P_n):n\in\om\}$; and finally $\mrm{tr}(\mc{N})=\mrm{Exh}(\psi)$ where $\psi(A)=\sum\{2^{-|s|}:s\in A$ is $\subseteq$-minimal in $A\}$ (and of course, $\varphi_h$, $\varphi_{\vec\vartheta}$, and $\psi$ are lsc submeasures).
\end{exa}

The following characterisation theorem gives us the most important tool when working on combinatorics of $F_\sigma$ ideals and analytic P-ideals.

\begin{thm}{\em (\cite{Mazur} and \cite{solecki})}\label{char} Let $\mc{I}$ be an ideal on $\om$.
\begin{itemize}
\item $\mc{I}$ is an $F_\sigma$ ideal iff $\mc{I}=\mrm{Fin}(\varphi)$ for some lsc submeasure $\varphi$.
\item $\mc{I}$ is an analytic $P$-ideal iff $\mc{I}=\mrm{Exh}(\varphi)$ for some lsc submeasure $\varphi$.
\item $\mc{I}$ is an $F_\sigma$ P-ideal iff $\mc{I}=\mrm{Fin}(\varphi)=\mrm{Exh}(\varphi)$ for some lsc submeasure $\varphi$.
\end{itemize}
\end{thm}

In particular, analytic P-ideals are $F_{\sigma\delta}$. When working with $\mrm{Exh}(\varphi)$, we can always assume that $\varphi(\{n\})>0$ for every $n$, and that $\varphi(\om)=\|\om\|_\varphi=1$: Let $\varphi_0(A)=\varphi(A)+\sum_{n\in A}2^{-n}$, $\varphi_1(A)=\min(\varphi_0(A),1)$, $\varphi_2(A)=\varphi_1(A)/\|\om\|_{\varphi_1}$, and $\varphi_3(A)=\min(\varphi_2(A),1)$, then $\mrm{Exh}(\varphi_i)=\mrm{Exh}(\varphi)$ for $i=0,1,2,3$ and $\varphi_3(\om)=\|\om\|_{\varphi_3}=1$.

\smallskip
As promised, we give an easy characterisation of nowhere tall analytic P-ideals:

\begin{fact}
Assume that $\mc{I}$ is a nowhere tall analytic P-ideal. Then $\mc{I}$ is a trivial modification of $\mrm{Fin}$ or $\mc{I}\simeq\{\0\}\otimes\mrm{Fin}$.
\end{fact}
\begin{proof}
Let $\mc{I}=\mrm{Exh}(\varphi)$ for some lsc submeasure $\varphi$. First we show that $\mc{I}$ is nowhere tall iff $\mc{I}=\mc{I}':=\{A\subseteq\om:A$ is finite or $\lim_{n\in A}\varphi(\{n\})=0\}$.

First assume that $\mc{I}$ is nowhere tall. As $\mc{I}\subseteq\mc{I}'$ always holds, we show that if $A\in\mc{I}'$ then $A\in\mc{I}$. Assume on the contrary, that $A\notin\mc{I}$. Then there is an infinite $A'\subseteq A$ such that $\mc{I}\clrest A'=[A']^{<\om}$. If $B=\{b_0<b_1<b_2<\dots\}\subseteq A'$ such that $\varphi(\{b_n\})<2^{-n}$ for every $n$, then $B\in\mc{I}$, a contradiction. Conversely, assume now that $\mc{I}=\mc{I}'$, and let $X\in\mc{I}^+$. If $Y\subseteq X$ is infinite and $\inf\{\varphi(\{k\}):k\in Y\}>0$ then $\mc{I}\clrest Y=[Y]^{<\om}$ and hence $\mc{I}\clrest X$ is not tall.

Therefore, if $\mc{I}$ is a nowhere tall analytic P-ideal, then there is a sequence $x_n>0$ such that $\mc{I}=\{A\subseteq\om:A$ is finite or $\lim_{n\in A}x_n=0\}$. By modifying $x_n$, we can assume that $x_n\ne x_m$ for every $n\ne m$. Let $X=\{x_n:n\in\om\}$, $X'$ be set of accumulation points of $X$, and $X''=(X')'$, we know that $X'\supseteq X''$ are closed. We have the following three cases: (1) $0\notin X'$. Then $\mc{I}=\mrm{Fin}$. (2) $0\in X'\setminus X''$. Let $y=\min(X'\setminus\{0\})>0$. Now $A\subseteq\om$ belongs to $\mc{I}$ iff $|A\cap \{n:x_n\geq y\}|<\om$, hence $\mc{I}$ is a trivial modification of $\mrm{Fin}$. (3) $0\in X''$. Fix a sequence $y_0=\infty>y_1>y_2\dots$ tending to $0$ such that each $[y_{k+1},y_k)$ contains infinitely many $x_n$, and let $P_k=\{n:x_n\in [y_{k+1},y_k)\}$. Then $A\in\mc{I}$ iff $A\cap P_k$ is finite for every $k$, and hence $\mc{I}\simeq\{\0\}\otimes\mrm{Fin}$.
\end{proof}

\section{Degrees of destruction}\label{degrees}
Starting with the usual forcing destructibility of ideals, we define three notions of destroying ideals in forcing extensions:
\begin{df}
Let $\mc{I}$ be an analytic or coanalytic ideal on $\om$, $\mc{D}=[\om]^\om,\mc{I}^+$, or $\mc{I}^*$, and let $\PP$ be a forcing notion. We say that {\em $\PP$ can $\mc{D}$-destroy $\mc{I}$} if there is a $p\in\PP$ such that $p\vd$``$\exists$ $Y\in \mc{D}$ $\forall$ $A\in\mc{I}^V$ $|X\cap A|<\om$''. Mostly we will write ($\infty$-)destroy, $+$-destroy, and $*$-destroy instead of $[\om]^\om/\mc{I}^+/\mc{I}^*$-destroy.
\end{df}
Clearly, $*$-destruction implies $+$-destruction which implies $\infty$-destruction. All these notions can be reformulated with pseudounions: $\PP$ destroys $\mc{I}$ if it adds a pseudounion of $\mc{I}\cap V$ with infinite complement, it $+$-destroys $\mc{I}$ if it adds a pseudounion of $\mc{I}\cap V$ with $\mc{I}$-positive complement, and $\PP$ $*$-destroys $\mc{I}$ if it adds a pseudounion of $\mc{I}\cap V$ with complement in $\mc{I}^*$, i.e. a pseudounion which belongs to $\mc{I}$.

\smallskip
Our main goals is to deepen our understanding of $\infty/+/*$-destructibility of Borel ideals, in particular, to characterise which ideals a fixed ``nice'' forcing notion $\PP$ can $\infty/+/*$-destroy and to study the associated cardinal invariants of these ideals.

\smallskip
Let us take a look on forcing (in)destructibility in the context of cardinal invariants. If $I$ is an ideal on $X$, then its {\em additivity, cofinality, uniformity, covering numbers} are defined as follows:
\begin{align*}
\add(I) & = \min\big\{|J|:J\subseteq I\;\text{and}\;\bigcup J\notin I\big\}\\
\cof(I) & = \min\big\{|D|:D\;\text{is cofinal in}\;(I,\subseteq)\big\}\\
\non(I) & = \min\big\{|Y|:Y\subseteq X\;\text{and}\;Y\notin I\big\}\\
\cov(I) & = \min\big\{|C|:C\subseteq I\;\text{and}\;\bigcup C=X\big\}
\end{align*}
In the case of ideals on countable underlying sets, most of these invariants equal $\om$. In \cite{cardinvanalp}, for tall ideals on $\om$ the following cardinal invariants were introduced:
\begin{align*}
\add^*(\mc{I})&= \min\big\{|\mc{U}|:\mc{U}\;\text{is unbounded in}\;(\mc{I},\subseteq^*)\big\}\\
\cof^*(\mc{I}) & = \min\big\{|\mc{D}|:\mc{D}\;\text{is cofinal in}\;(\mc{I},\subseteq^*)\big\}\\
\non^*(\mc{I}) & = \min\big\{|\mc{Y}|:Y\subseteq [\om]^\om\;\text{and}\;\forall\;A\in\mc{I}\;\exists\;Y\in \mc{Y}\;|A\cap Y|<\om\big\}\\
\cov^*(\mc{I}) & = \min\big\{|\mc{C}|:\mc{C}\subseteq \mc{I}\;\text{and}\;\forall\;Y\in [\om]^\om\;\exists\;C\in\mc{C}\;|Y\cap C|=\om\big\}
\end{align*}
Notice that in this context, destroying a Borel $\mc{I}$ means increasing $\cov^*(\mc{I})$, and similarly, $*$-destroying $\mc{I}$ is associated to increasing $\add^*(\mc{I})$. Strictly speaking, these invariants are not new, they can be seen as usual additivity, cofinality, uniformity, and covering numbers (see \cite{cardinvanalp}): If $\mc{I}$ is a tall ideal on $\om$, then let $\wh{\mc{I}}$ be the ideal on $[\om]^\om$ generated by all sets of the form $\wh{A}=\{X\in [\om]^\om:|A\cap X|=\om\}$, $A\in\mc{I}$. Now it is trivial to see that
\[ \mrm{inv}^*(\mc{I})=\mrm{inv}(\wh{\mc{I}})\;\,\text{for all four invariants above.}\]

To put $+$-destructibility into the context of cardinal invariants, we can easily generalise these cardinals by replacing $[\om]^\om$ with $\mc{I}^+$ in their definitions. In general, if $\mc{I}$ is an ideal on $\om$  then we define \[ \mrm{inv}^*(\mc{I},\mc{D}):=\mrm{inv}(\wh{\mc{I}}\clrest\mc{D})\;\text{where}\;
\mc{D}=[\om]^\om,\mc{I}^+,\mc{I}^*,\]
in particular, $\mrm{inv}^*(\mc{I})=\mrm{inv}^*(\mc{I},[\om]^\om)$. To avoid tedious notations, especially if the notation for an ideal or for its underlying set is too long (e.g. $\mrm{Conv}$ or $\om\times\om$), we will write $\mrm{inv}^*(\mc{I},\infty)=\mrm{inv}^*(\mc{I},[\om]^\om)$, $\mrm{inv}^*(\mc{I},+)=\mrm{inv}^*(\mc{I},\mc{I}^+)$, and $\mrm{inv}^*(\mc{I},*)=\mrm{inv}^*(\mc{I},\mc{I}^*)$.

When exactly are these cardinals defined? The invariants  $\add,\cof,\non,\cov$ are defined for proper ideals containing all finite subsets of their underlying sets. Properness is not an issue because if $\mc{I}$ is proper, then $\mc{I}^*\notin \wh{\mc{I}}$. Considering finite subsets of the underlying set, $[\om]^\om=\bigcup\wh{\mc{I}}$ iff $\mc{I}^+=\bigcup (\wh{\mc{I}}\clrest \mc{I}^+)$ iff $\mc{I}$ is tall; and $\mc{I}^*=\bigcup(\wh{\mc{I}}\clrest\mc{I}^*)$ iff $\mc{I}$ is not a trivial modification of $\mrm{Fin}$. To simply our list of conditions in the forthcoming statements, whenever we work with cardinal invariants of the form $\mrm{inv}^*(\mc{I},\infty)$ or $\mrm{inv}^*(\mc{I},+)$, we will always assume that $\mc{I}$ is tall. Similarly, when $\mrm{inv}^*(\mc{I},*)$ is involved, we will assume that $\mc{I}$ is not a trivial modification of $\mrm{Fin}$.

The cardinal invariants $\mrm{inv}^*(\mc{I},\infty)$ have been extensively studied (see e.g. \cite{cardinvanalp}, \cite{meza}, \cite{nwd}, and \cite{weakq}) but we know much less about e.g. $\mrm{inv}^*(\mc{I},+)$ (it was introduced in \cite{janajorg}). As we will mainly focus on uniformity and covering, let us reformulate these coefficients without referring to $\wh{\mc{I}}$:
\begin{align*}
\non^*(\mc{I},\infty) &=\min\big\{|\mc{Y}|:\mc{Y}\subseteq[\om]^\om,\;\forall\;A\in\mc{I}\;\exists\;Y\in\mc{Y}\;|A\cap Y|<\om\big\} \\
\cov^*(\mc{I},\infty) &=\min\big\{|\mc{C}|:\mc{C}\subseteq\mc{I},\;\forall\;Y\in[\om]^\om\;\exists\;C\in\mc{C}\;|Y\cap C|=\om\big\}\\
\non^*(\mc{I},+)&=\min\big\{|\mc{Y}|:\mc{Y}\subseteq\mc{I}^+,\;\forall\;A\in\mc{I}\;\exists\;Y\in\mc{Y}\;|A\cap Y|<\om\big\}\\
\cov^*(\mc{I},+) &=\min\big\{|\mc{C}|:\mc{C}\subseteq\mc{I},\;\forall\;Y\in\mc{I}^+\;\exists\;C\in\mc{C}\;|Y\cap C|=\om\big\}\\
\non^*(\mc{I},*)&=\min\big\{|\mc{Y}|:\mc{Y}\subseteq\mc{I}^*,\;\forall\;A\in\mc{I}\;\exists\;Y\in\mc{Y}\;|A\cap Y|<\om\big\}\\
\cov^*(\mc{I},*) &=\min\big\{|\mc{C}|:\mc{C}\subseteq\mc{I},\;\forall\;Y\in\mc{I}^*\;\exists\;C\in\mc{C}\;|Y\cap C|=\om\big\}
\end{align*}

\begin{obss}\label{cardobs}$ $
\begin{itemize}
\item[(1)]
Let $\mc{I}$ be an arbitrary tall ideal on $\om$ and $A,B\in\mc{I}$. Then the following are equivalent: (i) $A\subseteq^* B$, (ii) $\wh{A}\subseteq\wh{B}$, (iii) $\wh{A}\cap\mc{I}^+\subseteq\wh{B}\cap\mc{I}^+$, and (iv) $\wh{A}\cap\mc{I}^*\subseteq\wh{B}\cap\mc{I}^*$. (i)$\to$(ii)$\to$(iii)$\to$(iv) are trivial, and (iv) implies (i) because if $|A\setminus B|=\om$ then $\om\setminus B\in (\wh{A}\cap\mc{I}^*) \setminus\wh{B}$.
\item[(2)] Some of these new coefficients are actually equal:
\begin{align*}
\cof^*(\mc{I},\infty) & =\cof^*(\mc{I},+)=\cof^*(\mc{I},*)=\non^*(\mc{I},*),\\
\add^*(\mc{I},\infty) & =\add^*(\mc{I},+)=\add^*(\mc{I},*)=\cov^*(\mc{I},*).
\end{align*}
$\add^*(\mc{I},\infty) \leq \add^*(\mc{I},+)\leq \add^*(\mc{I},*)\leq \cov^*(\mc{I},*)$ are trivial (actually, (1) implies that the three additivities are equal), and $\cov^*(\mc{I},*)=\add^*(\mc{I},\infty)$ because an $\mc{A}\subseteq \mc{I}$ is $\subseteq^*$-unbounded in $\mc{I}$ iff $\forall$ $F\in\mc{I}^*$ $\exists$ $A\in\mc{A}$ $|F\cap A|=\om$. This argument can be ``dualised'', and we obtain the equalities $\cof^*(\mc{I},\infty) =\cof^*(\mc{I},+)=\cof^*(\mc{I},*)=\non^*(\mc{I},*)$.
\item[(3)] The remaining cardinal coefficients in a diagram (where $a\to b$ stands for $a\leq b$ and $\mc{D}=[\om]^\om,\mc{I}^+,\mc{I}^*$):
\begin{diagram}
 &  \non^*(\mc{I},\infty) & \rTo & \non^*(\mc{I},+)& \rTo & \non^*(\mc{I},*) & = \cof^*(\mc{I},\mc{D})\\
& \uTo & & & &  \uTo \\
\add^*(\mc{I},\mc{D}) = & \cov^*(\mc{I},*) & \rTo & \cov^*(\mc{I},+) & \rTo & \cov^*(\mc{I},\infty)
\end{diagram}
\item[(4)] If $\mc{I}$ is tall then $\cov^*(\mc{I},\infty)>\om$. If $\mc{I}$ is Borel and $\cov^*(\mc{I},+)=\om$, then no forcing notion can $+$-destroy $\mc{I}$. If $\mc{I}$ is Borel and $\cov^*(\mc{I},*)=\om$ (i.e. $\mc{I}$ is not a P-ideal), then no forcing notion can $*$-destroy $\mc{I}$ (see also in Section 1). The first statement is trivial. To show the second and third, notice that if $\mc{I}$ is Borel, then ``$(A_n)_{n\in\om}$ witnesses $\cov^*(\mc{I},+/*)=\om$'' is a $\Ubf{\Pi}^1_1$ property.

\item[(5)] If $\mc{I}\leq_\mrm{KB}\mc{J}$ are Borel, then (5a-i) $\cov^*(\mc{I},\infty)\geq\cov^*(\mc{J},\infty)$ and (5a-ii) $\non^*(\mc{I},\infty)\leq\non^*(\mc{J},\infty)$; and (5b-i) if $\PP$ cannot destroy $\mc{I}$ then $\PP$ cannot destroy $\mc{J}$ either, and dually, (5b-ii) if $\vd_\PP[\om]^\om\cap V\in\wh{\mc{I}}$ then $\vd_\PP[\om]^\om\cap V\in\wh{\mc{J}}$.

\item[(6)] If $\mc{I}\leq_\mrm{K}\mc{J}$ are Borel, then (5a-i) and (5b-i) hold, (6a)
  $\cov^*(\mc{I},+)\geq\cov^*(\mc{J},+)$ and $\non^*(\mc{I},+)\leq\non^*(\mc{J},+)$; and (6b) if $\PP$ cannot $+$-destroy $\mc{I}$ then $\PP$ cannot $+$-destroy $\mc{J}$ either, and dually, if $\vd_\PP\mc{I}^+\cap V\in\wh{\mc{I}}$ then $\vd_\PP\mc{J}^+\cap V\in\wh{\mc{J}}$.
\end{itemize}
\end{obss}
Point (6) above, more precisely the fact that K-reducibility is enough to obtain ``half'' of the consequences of KB-reducibility from point (5) is not so surprising after our remark on Kat\v{e}tov and Kat\v{e}tov-Blass reductions between our main examples (see after the definitions of these examples). Moreover if $\mc{I}\ne\mrm{Fin}$, $\mc{J}$ is a P-ideal, and $\mc{I}\leq_\mrm{K}\mc{J}$, then $\mc{I}\leq_\mrm{KB}\mc{J}$ holds as well: Fix a K-reduction $f:\om\to\om$ (that is, $f^{-1}[A]\in\mc{J}$ for every $A\in\mc{I}$) which is not finite-to-one and a $B\in\mc{J}$ such that $f^{-1}[\{n\}]\subseteq^* B$ for every $n$ (in particular, $B$ is infinite). Let $F_n=f^{-1}[\{n\}]\setminus B$ and fix an infinite element $A\in\mc{I}$. Define $g:\om\to\om$ such that $g\clrest F_n\equiv n$ for every $n$ and $g\clrest B$ is a bijection between $B$ and $A$. Then $g$ is a KB-reduction.

\smallskip
One may have noticed that $*$-destructibility, more precisely, the effect of reducibility between ideals on $*$-destructibility is missing from the list of our basic observations above. We will need a more general notion of reduction between ideals, see Section \ref{secstardest} for details.

\smallskip
We give a combinatorial characterisation of $\infty/+/*$-destructibility of Borel ideals by forcing notions of the form $\PP_I$. Unlike in the case of the original Theorem \ref{origchar}, we can work with arbitrary Polish spaces, do not have to understand trace ideals, and do not need continuous reading of names. Of course, this does not make this characterisation ``better'' (and it is certainly not ``deeper'') but it provides a new approach to forcing indestructibility of ideals.

In \cite{covprop}, based on a result from \cite{marci}, the authors introduced and studied the following notion: Let $X$ be an uncountable Polish space, $I$ be a $\sigma$-ideal on $X$, and $\mc{J}$ be an ideal on $\om$. Assuming that $X$ is clear from the context, we say that $I$ has the {\em $\mc{J}$-covering property}, $\mc{J}$-c.p. if for every {\em $I$-almost everywhere infinite-fold cover} $(B_n)_{n\in\om}$ of $X$ by Borel set, that is, $\{x\in X:\{n\in\om:x\in B_n\}$ is finite$\}\in I$, there is an $S\in \mc{J}$ (a ``small'' index set) such that $(B_n)_{n\in S}$ is still an $I$-a.e. infinite fold cover of $X$. It turned out that this property is a strong variant of forcing indestructibility: If $\PP_I$ is proper and $I$ has the $\mc{J}$-c.p., then $\PP_I$ cannot destroy $\mc{J}$. In general, the covering property is stronger:  $\mrm{Fin}\otimes\mrm{Fin}$ is Cohen-indestructible but $\mc{M}$ does not have the $\mrm{Fin}\otimes\mrm{Fin}$-c.p. See \cite{covprop} for more results about this property.

We show that a natural weak variant of the covering property is equivalent to forcing indestructibility, moreover, that the appropriate modifications work for $+/*$-indestructibility as well.

\begin{thm}\label{wcpfi} Let $\mc{J}$ be a Borel ideal on $\om$, $\mc{D}=[\om]^\om,\mc{J}^+$, or $\mc{J}^*$, and let $I$ be a $\sigma$-ideal on a Polish space $X$ such that $\PP_I$ is proper. Then $\PP_I$ cannot $\mc{D}$-destroy $\mc{J}$ if, and only if for every sequence $(B_n)_{n\in\om}$ of Borel subsets of $X$
\begin{align*}
\textbf{IF} &\;\;\;\big\{x:\big\{n\in\om:x\in B_n\}\in\mc{D}\big\}\in I^+,\\ \textbf{THEN} &\;\;\;\big\{x:\big|\big\{n\in S:x\in B_n\big\}\big|=\om\big\}\in I^+\;\;\text{for some}\;S\in\mc{J}.
\end{align*}
\end{thm}
\begin{proof}
Let $\mc{D}=[\om]^\om$ or $\mc{J}^+$ or $\mc{J}^*$ accordingly.

The ``if'' direction: Assume on the contrary that $\Vdash_{\mathbb{P}_I}\mr{Y}\in \mc{D}$ and $C\Vdash\forall$ $A\in\mc J^V$ $|\mr{Y}\cap A|<\om$ for some $C\in\mathbb{P}_I$. Applying the {\em Borel reading of names} (see \cite[Prop. 2.3.1]{zap}), there are a $C'\in\mathbb{P}_I$, $C'\leq C$ and a Borel function $f:C'\to \mc{D}$ (coded in the ground model) such that $C'\Vdash_{\mathbb{P}_I}f(\mr{r}_I)=\mr{Y}$ where $\mr{r}_I$ is the generic real. For $n\in \om$ define $B_n=\{x\in C':n\in f(x)\}$. Then $B_n=f^{-1}[\{A\subseteq \om:n\in A\}]$ is Borel and $C'=\{x\in X:\{n\in\om:x\in B_n\}\in\mc{D}\}\in I^+$. Applying our assumption, there is an $S\in\mc{J}$ such that $C''=\{x\in C':|\{n\in S:x\in B_n\}|=\om\}\in I^+$. Notice that $C''=\{x\in C':|f(x)\cap S|=\om\}=f^{-1}[\{A\subseteq\om:|A\cap S|=\om\}]$ is also Borel and hence it is a condition below $C'$, and of course $C''\Vdash_{\mathbb{P}_I}|f(\mr{r}_I)\cap S|=\om$, a contradiction.

\smallskip
The ``only if'' direction: Let $(B_n)_{n\in\om}$ be a sequence of Borel subsets of $X$ such that $\{x:\{n\in\om:x\in B_n\}\in\mc{D}\}=C\in I^+$. Notice that $C$ is Borel because $g:X\to \mc{P}(\om)$, $g(x)=\{n\in\om:x\in B_n\}$ is a Borel function and $C=g^{-1}[\mc{D}]$. In particular, $C\in\PP_I$, $C\vd \mr{Y}:=\{n\in\om:\mr{r}_I\in B_n\}\in\mc{D}$, and hence (as $\PP_I$ cannot $\mc{D}$-destroy $\mc{J}$) there are a $C'\leq C$ and an $S\in\mc{J}$ such that $C'\vd \mr{Y}\cap S=|\{n\in S:\mr{r}_I\in B_n\}|=\om$, equivalently, $\{x\in C':|\{n\in S:x\in B_n\}|<\om\}\in I$, and hence  $\{x:|\{n\in S:x\in B_n\}|=\om\}\in I^+$.
\end{proof}

\begin{rem}
One may wonder what the exact role of CRN was in the original characterisation of forcing indestructibility of ideals (see Theorem \ref{origchar}). The ideal $I$ satisfies the {\em continuous reading of names}, if for every Polish space $Y$, $B\in\PP_I$, and Borel $g:B\to \,\!^\om 2$, there is a $C\in\PP_I$, $C\subseteq B$ such that $g\clrest C$ is continuous (in particular, the function $f$ in the application of Borel reading of names above can be chosen as continuous). For example, the following properties imply CRN: (a) $\PP_I$ is   $\,\!^\om\om$-bounding (e.g. $\mc{N}$ and $[\,\!^\om 2]^{\leq\om}$, i.e. the random and Sacks forcings); (b) $I$ is $\sigma$-generated by closed sets (e.g. $\mc{M}$ and $\mc{K}_\sigma$, i.e. the Cohen and Miller forcings); (c) the ``natural'' ideal generating the Hechler forcing $\mbb{D}$; (d) the ideal generating the Laver forcing $\mbb{L}$; (e) the ideal generating $\mbb{M}(\mc{I}^*)$ if $\mc{I}$ is a P-ideal; on the other hand, regardless its presentation, the eventually different real forcing does not have the CRN (for more details see \cite[Section 3.1]{zap} and \cite{quot}).

We know that under CRN the following are equivalent: (a) $\PP_I$ cannot destroy $\mc{J}$, and (b) $\mc{J}\nleq_\mrm{K}\mrm{tr}(I)\clrest X$ for every $X\in\mrm{tr}(I)^+$. We show that without assuming CRN, (a) still implies (b). Assume on the contrary that $f:X\to \om$ witnesses $\mc{J}\leq_\mrm{K} \mrm{tr}(I)\clrest X$ for some $X\in\mrm{tr}(I)^+$. Then $[X]_\delta$ is an $I$-positive $G_\delta$ set. Define the Borel sets $B_n=\{x\in [X]_\delta: \exists$ $k$ $(x\clrest k\in X$ and $f(x\clrest k)=n)\}$. Now if $x\in [X]_\delta$, then $|\{n\in\om:x\in B_n\}|<\om$ iff $f[\{x\clrest k:k\in\om\}\cap X]\subseteq m$ for some $m$, in particular, $x\in\bigcup_{m\in\om} [f^{-1}[m]]_\delta\in I$, and hence $\{x\in [X]_\delta:|\{n\in\om:x\in B_n\}|=\om\}\in I^+$. Applying Theorem \ref{wcpfi}, there is an $S\in\mc{J}$ such that $\{x\in [X]_\delta:|\{n\in S:x\in B_n\}|=\om\}\in I^+$ but $\{x\in [X]_\delta:|\{n\in S:x\in B_n\}|=\om\}=[f^{-1}[S]]_\delta\in I$, a contradiction.
\end{rem}

\section{Examples}\label{secexa}

In this section, we discuss some of our main examples $\mc{I}$, their cardinal invariants $\non^*(\mc{I},+)$ and $\cov^*(\mc{I},+)$, and their ($+$-)destructibility. For a survey on the invariants $\mrm{inv}^*(\mc{I},\infty)$, see e.g. \cite{meza}, \cite{weakq} (for $\mc{ED}$ and $\mc{ED}_\mrm{fin}$), and \cite{nwd} (for $\mrm{Nwd}$).

\subsection*{Easy examples: $\mrm{Fin}\otimes\mrm{Fin}$, $\mrm{Conv}$, and $\mrm{Ran}$}

\begin{exa} When working with cardinal invariants of $\mrm{Fin}\otimes\mrm{Fin}$ we will write $\mrm{Fin}^2=\mrm{Fin}\otimes\mrm{Fin}$. We know that $\non^*(\mrm{Fin}^2,\infty)=\om$, $\cov^*(\mrm{Fin}^2,\infty)=\mf{b}$, and $\cof^*(\mrm{Fin}^2,\infty)=\mf{d}$; and it is easy to show the following: \begin{itemize}
\item[(1)] $\non^*(\mrm{Fin}^2,+)=\mf{d}$ and $\cov^*(\mrm{Fin}^2,+)=\om$.
\item[(2a)] $\PP$ destroys $\mrm{Fin}\otimes\mrm{Fin}$ iff $\PP$ adds dominating reals.
\item[(2b)] No forcing notion can $+$-destroy $\mrm{Fin}\otimes\mrm{Fin}$.
\end{itemize}
\end{exa}
\begin{exa}
We know that $\non^*(\mrm{Conv},\infty)=\om$ and $\cov^*(\mrm{Conv},\infty)=\mf{c}$. We show the following:
\begin{itemize}
\item[(1)] $\non^*(\mrm{Conv},+)=\om$ and $\cov^*(\mrm{Conv},+)=\mf{c}$.
\item[(2)] If a forcing notion adds new reals then it $+$-destroys $\mrm{Conv}$.
\end{itemize}

(1): A countable base of the topology of $\mbb{Q}$ witnesses the uniformity. Now if $x^\al_n\to y_\al$ are convergent sequences, $x^\al_n\in\mbb{Q}$ and $\al<\ka<\mf{c}$, then there is a (nontrivial) convergent sequence $z_n\to z$, $z_n,z\in[0,1]\setminus\{y_\al:\al<\ka\}$, and if $\mbb{Q}\ni r^n_k\xrightarrow{k\to\infty}z_n$ such that $R=\{r^n_k:n,k\in\om\}$ has no accumulation points apart from the $z_n$'s and $z$, then $R\in\mrm{Conv}^+$ witnesses that $\{\{x^\al_n:n\in\om\}:\al<\om\}$ cannot be a covering family.

(2): Notice that if $\PP$ adds a new real then it adds a (nontrivial) sequence $(z_n)_{n\in\om}$ of new reals converging to a new real $z$, and hence the argument above shows that $\PP$ $+$-destroys $\mrm{Conv}$.
\end{exa}

\begin{exa}
We know that $\non^*(\mrm{Ran},\infty)=\om$ and $\cov^*(\mrm{Ran},\infty)=\mf{c}$; and applying $\mrm{Ran}\leq_\mrm{KB}\mrm{Conv}$ (see e.g. \cite{hrusakkatetov}), the last example, and Observations \ref{cardobs} (5) and (6), we know that
\begin{itemize}
\item[(1)] $\non^*(\mrm{Ran},+)=\om$ and $\cov^*(\mrm{Ran},+)=\mf{c}$;
\item[(2)] if a forcing notion adds new reals then it $+$-destroys $\mrm{Ran}$.
\end{itemize}
\end{exa}

\begin{prob}
Can we characterise those Borel ideals which are ($+$-)destroyed by every forcing notion introducing a new real? (Loosely speaking, we would like to characterise those Borel ideals $\mc{I}$ such that $\mrm{ZFC}$ proves $\cov^*(\mc{I},\infty/+)=\mf{c}$.) Does there exist a tall Borel ideal $\mc{I}$ which is destroyed by every forcing notion introducing a new real but $\mc{I}$ is not $+$-destroyed by all these forcing notions?
\end{prob}

\subsection*{Around $\mc{ED}$ and $\mc{ED}_\mrm{fin}$}

We know (see \cite{weakq}) that $\non^*(\mc{ED},\infty)=\om$, $\cov^*(\mc{ED},\infty)=\non(\mc{M})$, and $\cof^*(\mc{ED},\infty)=\mf{c}$.
\begin{prop}$ $
\begin{itemize}
\item[(1a)] $\non^*(\mc{ED},+)=\cov(\mc{M})$;
\item[(1b)] $\cov^*(\mc{ED},+)=\non(\mc{M})$;
\item[(2)] $\PP$ $+$-destroys $\mc{ED}$ iff $\PP$ destroys $\mc{ED}$ iff $\PP$ adds an eventually different real (that is, an $f\in\,\!^\om\om$ such that $|f\cap g|<\om$ for every $g\in\,\!^\om\om\cap V$).
\end{itemize}
\end{prop}
\begin{proof}
(1a): Let us recall the following characterisations of $\cov(\mc{M})$ (see \cite[Lem. 2.4.2, Thm. 2.4.5]{BaJu}):
Let $\mc{C}=\{S\in \,\!^\om ([\om]^{<\om}):\sum_{n\in\om}|S(n)|/n^2<\infty\}$ (for now let $1/0=1$). Then
\[
\cov(\mc{M}) =\min\big\{|F|:F\subseteq\,\!^\om\om\;\forall\;S\in\mc{C}\;\exists\;f\in F\;\forall^\infty\;n\;f(n)\notin S(n)\big\} \]
\[=\min\big\{|\mc{D}|:\forall\;D\in\mc{D}\;(D\subseteq\mbb{C}\;\text{is dense})\;\text{and}\;\nexists\;\mc{D}\text{-generic filter}\;G\subseteq\mbb{C}\big\}.\]
To show $\non^*(\mc{ED},+)\leq\cov(\mc{M})$, fix an $F\subseteq\,\!^\om\om$ witnessing the above characterisation. For every $f\in F$ define $X_f\in\mc{ED}^+$ as follows
\[X_f=\big\{(n,k)\in \om\times\om:f(n)-\lfloor\sqrt{n}\rfloor\leq k\leq f(n)+\lfloor\sqrt{n}\rfloor\big\}.\]
Let $A\in\mc{ED}$, we show that there is an $f\in F$ such that $|A\cap X_f|<\om$, and hence $\non^*(\mc{ED},+)\leq|F|=\cov(\mc{M})$. As columns have finite intersection with every $X_f$, we can assume that $A$ is of the form $\bigcup\{\{n\}\times F_n:n\in \om\}$ where $F_n\in [\om]^m$ for some fixed $m\in\om$. Let $S_A:\om\to [\om]^{<\om}$, \[ S_A(n)=\bigcup_{k\in F_n}\big[k-\lfloor\sqrt{n}\rfloor, k+\lfloor\sqrt{n}\rfloor\big]\cap\om.\]
Then $S_A\in\mc{C}$ and hence there is an $f\in F$ such that $f(n)\notin S_A(n)$ for almost all $n$, and so $(\{n\}\times F_n)\cap X_f=\0$ for almost all $n$, i.e. $|A\cap X_f|<\om$.

Conversely, we show that if a family $\{X_\al:\al<\ka\}$ witnesses $\non^*(\mc{ED},+)$ then there is a family $\mc{D}$ of dense subsets of $\mbb{C}$, $|\mc{D}|=\ka$ such that no filter on $\mbb{C}$ is $\mc{D}$-generic. We know that for every $\al$ there are infinitely $k$ such that $(X_\al)_k\ne\0$. Interpret $\mbb{C}$ now as $(\,\!^{<\om}\om,\supseteq)$ and for every $\al$ and $n$ let $D_{\al,n}=\{s\in \mbb{C}:\exists$ $k\geq n$ $s(k)\in (X_\al)_k\}$. Then $D_{\al,n}$ is dense in $\mbb{C}$, and if $G\subseteq \mbb{C}$ is a $\{D_{\al,n}:\al<\ka,n\in\om\}$-generic filter, then $g=\bigcup G:\om\to\om$ (in particular, $g\in \mc{ED}$) and $|g\cap X_\al|=\om$ for every $\al$, a contradiction.

(1b): We already know that $\cov^*(\mc{ED},+)\leq\cov^*(\mc{ED},\infty)=\non(\mc{M})$. To show the reverse inequality, we will need the following characterisation (see \cite[Lem. 2.4.8]{BaJu}):
\[ \non(\mc{M})=\min\big\{|\mc{S}|:\mc{S}\subseteq\mc{C}\;\text{and}\;\forall\;f\in\,\!^\om\om\;\exists\;S\in\mc{S}\;\exists^\infty\;n\;f(n)\in S(n)\big\}.\]
Notice that there is a family witnessing $\cov^*(\mc{ED},+)$ of the form  $\{\{n\}\times\om:n\in\om\}\cup\{f_\al:\al<\ka\}$ where $f_\al:\om\to\om$. For every $\al$ define $S_\al\in\mc{C}$, $S_\al(n)=(X_{f_\al})_n=\big[f_\al(n)-\lfloor\sqrt{n}\rfloor, f_\al(n)+\lfloor\sqrt{n}\rfloor\big]\cap\om$. Using the same argument we used in (1a), one can easily show that $\{S_\al:\al<\ka\}$ satisfies the conditions in the above characterisation of $\non(\mc{M})$, and hence $\non(\mc{M})\leq\ka$.

(2): The first ``left to right'' implication is trivial. The second one is basically \cite[Lem. 2.4.8, (2)$\to$(3)]{BaJu}. Assume that in an extension $W\supseteq V$ there is an $A\in [\om\times\om]^\om$ such that $|A\cap B|<\om$ for every $B\in\mc{ED}\cap V$. By shrinking $A$ can assume that $A=\{(n,k_n):n\in E\}$, $E\in [\om]^\om$ is an infinite partial function. Let $E=\{n_0<n_1<\dots\}$, $\mrm{FP}=\{$finite partial functions $\om\to\om\}$, and let $f\in\,\!^\om\mrm{FP}\cap W$, $f(m)=\{(n_i,k_{n_i}):i\leq m\}$. We show that $f$ is an eventually different real over $\,\!^\om\mrm{FP}\cap V$. Let $g\in\,\!^\om\mrm{FP}\cap V$ and assume on the contrary that $f(m)=g(m)$ for infinitely many $m$. We can assume that $|\dom(g(m))|=m+1$ for every $m$. Define the infinite partial function $g'\in V$ by recursion as follows: Let $\dom(g(0))=\{m_0\}$ and $g'(m_0)=g(0)(m_0)$. If we already have $m_0,m_1,\dots,m_{n-1}$ and $g'$ is defined on these entries, then pick an $m_n\in\dom(g(n))\setminus\{m_0,m_1,\dots,m_{n-1}\}$ and define $g'(m_n)=g(n)(m_n)$. It is trivial to show that $|A\cap g'|=\om$, a contradiction.

Finally we show that if $f\in\,\!^\om\om\cap W$ is an eventually different real over $V$ then $\mc{ED}\cap V$ is $+$-destroyed in $W$. Fix an interval partition $(P_n)_{n\in\om}$ in $V$ such that $|P_n|=n+1$, and fix enumerations $\{(a^n_i,b^n_i):i\in\om\}=P_n\times\om$. Define $X=\{(n,i):f(a^n_i)=b^n_i\}\in \mc{ED}^+\cap W$ (because $|(X)_n|=n+1$). We claim that $X$ $+$-destroys $\mc{ED}\cap V$. Let $g\in\,\!^\om\om\cap V$ and assume on the contrary that $|X\cap g|=\om$. Define $g'\in\,\!^\om\om\cap V$ as follows: If $g(n)=i$ then let $g' \clrest P_n\equiv b^n_i$. It follows that $f(a)=g'(a)$ for infinitely many $a$, a contradiction.
\end{proof}

Cardinal invariants of $\mc{ED}_\mrm{fin}$ are more intriguing (see \cite{weakq}): $\add^*(\mc{ED}_\mrm{fin},\infty)=\om$, $\cof^*(\mc{ED}_\mrm{fin},\infty)=\mf{c}$, $\mf{s}\leq\cov^*(\mc{ED}_\mrm{fin},\infty)$, $\non^*(\mc{ED}_\mrm{fin},\infty)\leq\mf{r}$ (where $\mf{s}$ and $\mf{r}$ are the {\em slitting} and {\em reaping} numbers), furthermore, $\cov(\mc{M})=\min\{\mf{d},\non^*(\mc{ED}_\mrm{fin},\infty)\}$, and $\non(\mc{M})=\max\{\cov^*(\mc{ED}_\mrm{fin},\infty),\mf{b}\}$.

\begin{prop}\label{edfin}$ $
\begin{itemize}
\item[(1)] $\non^*(\mc{ED}_\mrm{fin},+)=\non^*(\mc{ED}_\mrm{fin},\infty)$ and $\cov^*(\mc{ED}_\mrm{fin},+)=\cov^*(\mc{ED}_\mrm{fin},\infty)$;
\item[(2)] $\PP$ $+$-destroys $\mc{ED}_\mrm{fin}$ iff $\PP$ destroys $\mc{ED}_\mrm{fin}$ iff $\PP$ adds an eventually different infinite partial function $f\subseteq\Delta$ iff $\PP$ adds an eventually different infinite partial function bounded by a ground model real.
\end{itemize}
\end{prop}
\begin{proof}
(1): We know that $\non^*(\mc{ED}_\mrm{fin},+)\geq \non^*(\mc{ED}_\mrm{fin},\infty)$ and $\cov^*(\mc{ED}_\mrm{fin},+)\leq\cov^*(\mc{ED}_\mrm{fin},\infty)$.  For every $n\in\om$ fix a partition $(\Delta)_{(n+1)^2-1}=\{((n+1)^2-1,i):i< (n+1)^2\}=\bigcup_{k\leq n}P^n_k$ such that $|P^n_k|=n+1$ for every $k$, and define the following functions:
\begin{itemize}
\item[(i)] $f:\Delta\to [\Delta]^{<\om}$, $f(n,k)=P^n_k$;
\item[(ii)] $\al:[\Delta]^\om\to\mc{ED}_\mrm{fin}^+$, $\al(X)=\bigcup\{f(n,k):(n,k)\in X\}$;
\item[(iii)] $\be:\mc{ED}_\mrm{fin}\to\mc{ED}_\mrm{fin}$,  $\be(A)=\{(n,k):f(n,k)\cap A\ne\0\}$.
\end{itemize}
Now, if $\mc{X}\subseteq [\Delta]^\om$ witnesses $\non^*(\mc{ED}_\mrm{fin},\infty)$ then $\al[\mc{X}]=\{\al(X):X\in\mc{X}\}$ witnesses $\non^*(\mc{ED}_\mrm{fin},+)$: Otherwise, if $A\in\mc{ED}_\mrm{fin}$ and $|A\cap \al(X)|=\om$ for every $X\in\mc{X}$, then $|\be(A)\cap X|=\om$ for every $X\in\mc{X}$, a contradiction. Similarly, if $\mc{A}\subseteq\mc{ED}_\mrm{fin}$ witnesses $\cov^*(\mc{ED}_\mrm{fin},+)$ then $\be[\mc{A}]$ witnesses $\cov^*(\mc{ED}_\mrm{fin},\infty)$: Otherwise, if $Y\in [\Delta]^\om$ has finite intersection with all $\be(A)$, then $\al(Y)\in\mc{ED}_\mrm{fin}^+$ has finite intersection with all $Y\in\mc{A}$, a contradiction.

(2): All ``left to right'' implications are trivial. Assume now that $V\subseteq W$ is an extension and $g\in W$ is an eventually different infinite partial function over $V$, $g\leq h\in\om^\om\cap V$. We can assume that $h$ is strictly increasing. It is straightforward to show that $X=\{(h(n),g(n)):n\in\om\}\subseteq\Delta$ is also an eventually different infinite partial function over $V$, and hence $\al(X)$ $+$-destroys $\mc{ED}_\mrm{fin}\cap V$.
\end{proof}

\subsection*{Around $\mc{S}$}
We know that $\non^*(\mc{S},\infty)=\om$, $\cov^*(\mc{S},\infty)=\non(\mc{N})$, and $\cof^*(\mc{S},\infty)=\mf{c}$.
\begin{prop} {\em (basically \cite[Thm. 1.6.2]{meza})}
\begin{itemize}
\item[(1a)] $\non^*(\mc{S},+)=\om$;
\item[(1b)] $\cov^*(\mc{S},+)=\non(\mc{N})$;
\item[(2)] $\PP$ $+$-destroys $\mc{S}$ iff $\PP$ destroys $\mc{S}$ iff $\vd_\PP \,\!^\om 2\cap V\in\mc{N}$.
\end{itemize}
\end{prop}
\begin{proof}
(1a): Let $\mc{F}=\{F\in [\,\!^{<\om}2]^{<\om}\setminus\{\0\}:\sum_{t\in F} 2^{-|t|}\leq 1/4\}$ and for every $F\in\mc{F}$ define the clopen set $U_F=\bigcup_{t\in F}\{x\in\,\!^\om 2:t\subseteq x\}$ and the family $\mc{U}_F=\{C\in\Omega:C\cap U_F=\0\}$. Notice that $\mc{U}_F\in\mc{S}^+$ because if $X\subseteq \,\!^\om 2$ is finite then $U_F\cup X$ is a closed set of measure $\leq 1/4<1/2$ and hence there is a clopen set $C$ of measure $1/2$ inside its complement, therefore $\mc{U}_F\nsubseteq \bigcup_{x\in X}\mc{C}_x$. We show that for every $\mc{A}\in\mc{S}$ there is an $F\in\mc{F}$ such that $\mc{A}\cap\mc{U}_F=\0$, i.e. that $\{\mc{U}_F:F\in\mc{F}\}$ witnesses $\non^*(\mc{S},+)=\om$. Let $\{x_i:i<k\}\subseteq \,\!^\om 2$ be finite. Pick finite initials $t_i\subseteq x_i$ such that $\sum_{i<k}2^{-|t_i|}\leq 1/4$ and let $F=\{t_i:i<k\}\in\mc{F}$, then $\mc{U}_F\cap \bigcup_{i<k}\mc{C}_{x_i}=\0$ (because $x_i\in U_F$ for every $i$, and $C\cap U_F=\0$ for every $C\in\mc{U}_F$).

(1b): Let $\lam^*$ be the Lebesgue outer measure on $\,\!^\om 2$. We show that if $\lam^*(Y)<1/2$ then there is an $\mc{S}$-positive $\mc{D}\subseteq\Omega$ such that $|\mc{C}_y\cap \mc{D}|<\om$ for every $y\in Y$. This implies that $\non(\mc{N})\leq\cov^*(\mc{S},+)$, and the reverse inequality follows from $\cov^*(\mc{S},+)\leq\cov^*(\mc{S},\infty)=\non(\mc{N})$.

Fix an increasing sequence of clopen sets $U_n$ such that $Y\subseteq \bigcup_{n\in\om}U_n$ and the measure of this union is less then $1/2-\eps$ for some $\eps>0$. Enumerate $\{V_n:n\in\om\}$ all clopen sets of measure $<\eps$ and for each $n$ pick a $C_n\in\Omega$ such that $C_n\cap (U_n\cup V_n)=\0$ (this is possible because $U_n\cup V_n$ is a closed set of measure $<1/2$). The set $\mc{D}=\{C_n:n\in\om\}$ is $\mc{S}$-positive because if $X\subseteq \,\!^\om 2$ is finite, then $X\subseteq V_n$ for some $n$, hence $C_n\notin \bigcup_{x\in X} \mc{C}_x$ (and so $\mc{D}\nsubseteq \bigcup_{x\in X} \mc{C}_x$). Also, if $y\in Y$, then $y\in U_n$ in particular $y\notin C_n$ for every large enough $n$, and hence $|\mc{D}\cap \mc{C}_y|<\om$.

(2): The first ``only if'' implication is trivial.

Now assume that $\PP$ destroys $\mc{S}$. We will need the following result (see \cite[Lem. 1.6.3 (b)]{meza}): If $\lam^*(Y)>1/2$ then for every infinite $\mc{D}\subseteq\Omega$ there is a $y\in Y$ such that $|\mc{D}\cap\mc{C}_y|=\om$. This implies that $\vd_\PP\lam^*(\,\!^\om 2\cap V)\leq 1/2$. Notice that $\vd_\PP$``$\lam^*(\,\!^\om 2\cap V)=0$ or $1$'' holds for every $\PP$: If $V[G]\models \lam^*(\,\!^\om 2\cap V)<1$ then there is a compact set $C\in V[G]$ of positive measure which is disjoint from $V$, and hence, applying the $0$-$1$ law, the $F_\sigma$ tail-set $\{x\in \,\!^\om 2:\exists$ $y\in C$ $|x\vartriangle y|<\om\}$ generated by $C$ is of measure $1$, and of course, this set is also disjoint from $V$. We conclude that $\vd_\PP \,\!^\om 2\cap V\in\mc{N}$.

Finally, the last implication follows from the result we used in (1b).
\end{proof}

\subsection*{Around $\mrm{Nwd}$}

We know (see \cite{nwd}) that $\non^*(\mrm{Nwd},\infty)=\om$, $\cov^*(\mrm{Nwd},\infty)=\cov(\mc{M})$, and $\cof^*(\mrm{Nwd},\infty)=\cof(\mc{M})$.

\begin{prop}$ $
\begin{itemize}
\item[(1)] {\em (see \cite{kerem} and \cite{nwd})} $\non^*(\mrm{Nwd},+)=\om$ and $\cov^*(\mrm{Nwd},+)=\add(\mc{M})$.
\item[(2)] If $\PP$ adds Cohen reals then it destroys $\mrm{Nwd}$. If $\PP$ $+$-destroys $\mrm{Nwd}$ then it adds both dominating and Cohen reals.
\item[(2-cd)] {\em (see \cite{nwd})} If $\PP$ adds a Cohen real and $\vd_\PP$``$\QQ$ adds a dominating real'', then $\PP\ast\QQ$ $+$-destroys $\mrm{Nwd}$.
\item[(2-dc)] Adding first a dominating then a Cohen real does not necessarily $+$-destroy $\mrm{Nwd}$: If $\PP$ has the Laver property then $\PP$ cannot destroy $\mrm{Nwd}$ and $\PP\ast \mbb{C}$ cannot $+$-destroy $\mrm{Nwd}$.
\end{itemize}
\end{prop}
\begin{proof}
(1): A countable base of the topology witnesses $\non^*(\mrm{Nwd},+)=\om$, and reformulating a result from \cite{kerem} (see also in \cite{nwd}) shows that $\cov^*(\mrm{Nwd},+)=\add(\mc{M})$.

(2): We already know that $\mbb{C}=\PP_\mc{M}$ destroys $\mrm{tr}(\mc{M})\simeq\mrm{Nwd}$, and hence adding a Cohen real destroys $\mrm{Nwd}$.

First we show that if $\PP$ $+$-destroys $\mrm{Nwd}$, then $\PP$ adds a dominating real. For now let $\mrm{Nwd}=\{A\subseteq\,\!^{<\om} 2:\forall$ $s$ $\exists$ $t$ $(s\subseteq t$ and $A\cap t^\uparrow=\0)\}$, and let $\mr{X}$ be a $\PP$-name such that $\vd_\PP\mr{X}\in\mrm{Nwd}^+$ and $\vd_\PP|\mr{X}\cap A|<\om$ for every $A\in\mrm{Nwd}\cap V$. We can assume that $\mr{X}$ is dense in $\,\!^{<\om} 2$, that is, $\vd_\PP\forall$ $s$ $\exists$ $t$ $s\subseteq t\in \mr{X}$ (because there is a $\PP$-name $\mr{t}$ for a node in $\,\!^{<\om} 2$ such that $\vd_\PP$``$\mr{X}$ is dense in $\mr{t}^\uparrow$'' and the proof below can be easily modified to $\mr{t}^\uparrow\simeq\,\!^{<\om} 2$). Let $\mr{f}$ be a $\PP$-name for an element of $\,\!^\om\om$ such that
\[ \vd_\PP\mr{f}(n)=\max\big\{\min\big\{|t|:s\subseteq t\in \mr{X}\big\}:s\in\,\!^n 2\big\}\;\;\text{for every}\;n.\]
We claim that $\mr{f}$ is dominating over $\,\!^\om\om\cap V$: Let $g\in\,\!^\om\om\cap V$ be strictly increasing and satisfying $g(0)>1$. Fix an infinite maximal antichain $\{a_n:n\in\om\}\subseteq\,\!^{<\om} 2$ such that $|a_n|=g(n)$, and let $A=\,\!^{<\om} 2\setminus\bigcup\{a_n^\uparrow:n\in\om\}\in\mrm{Nwd}$. It is easy to see that $|A\cap\,\!^n 2|\geq 2^{n-1}$ for every $n$. We know that in the extension $A\cap\mr{X}\subseteq \,\!^{\leq N} 2$ for an $N\in\om$. Now if $n>N$ then we can pick a point $s\in A\cap \,\!^n 2$ such that $s\nsubseteq a_k$ for $k<n$. As there can be no $a_k$ below $s$, $K=\min\{k:s\subseteq a_k\}<\om$, and $|a_K|>K\geq n$. In particular, $s^\uparrow\cap \,\!^{<|a_K|} 2=\{t:s\subseteq t$ and $|t|<|a_K|\}\subseteq A$, and hence $\mr{f}(n)\geq\min\{|t|:s\subseteq t\in \mr{X}\}\geq |a_K|=g(K)\geq g(n)$.

Now we show that if $\PP$ $+$-destroys $\mrm{Nwd}$, then it adds Cohen reals. Let $\mr{X}$ be as above, then in the extension $[\mr{X}]_\delta\ne\0$. We show that every element $y$ of this set is Cohen over $V$: If $C\subseteq \,\!^\om 2$ is a closed and nowhere dense set coded in $V$, then there is a tree $T\subseteq \,\!^{<\om}2$, $T\in\mrm{Nwd}$ such that $C=[T]_\delta=[T]:=\{x\in\,\!^\om 2:\forall$ $n$ $x\clrest n\in T\}$, in particular $|\mr{X}\cap T|<\om$ and hence $y\notin [T]$.

(2-cd): This is basically \cite[Thm. 1.4 (ii)]{nwd}. Let $G$ be $(V,\PP)$-generic, $c\in \,\!^\om 2\cap V[G]$ be Cohen over $V$, $H$ be $(V[G],\QQ[G])$-generic, and $d\in\,\!^\om\om\cap V[G,H]$ be dominating over $V[G]$. Enumerate $\,\!^{<\om} 2=\{t_n:n\in\om\}$ in $V$ and for every $n$ define $c_n\in \,\!^\om 2\cap V[G]$ as $c_n=t_n^\frown (c(k):k\geq |t_n|)$, then $c_n$ is also Cohen over $V$, and let $X=\{c_n\clrest m:m\geq d(n)\}\in V[G,H]$. Notice that $X$ is dense in $\,\!^{<\om} 2$. Now if $A\in\mrm{Nwd}\cap V$ then $|A\cap \{c_n\clrest m:m\in\om\}|<\om$, in particular
\[ f_A(n)=\min\big\{k:\forall\;m\geq k\;c_n\clrest m\notin A\big\}\]
is well defined, and $f_A\in \,\!^\om\om\cap V[G]$. We know that $f_A(n)\leq d(n)$ for every $n\geq N_A$ for some $N_A\in \om$, and hence $X\cap A\subseteq \{c_n\clrest m:n<N_A,m<f_A(n)\}$.

(2-dc): If $\mr{X}$ is a $\PP$-name, $p\in\PP$, and $p\vd\mr{X}\in [\,\!^{<\om} 2]^\om$, then we can assume that $\mr{X}$ is either (i) an infinite chain, that is,  $p\vd$``$\mr{X}=\{\mr{s}_0\subsetneq \mr{s}_1\subsetneq\dots\}$ and $\mr{s}_k\subseteq\mr{x}\in\,\!^\om 2$ for every $k$'', or  (ii) a converging antichain, that is, there is a $\PP$-name $\mr{y}$ such that $p\vd$``$\mr{y}\in \,\!^\om 2$ and $\forall$ $n$ $\forall^\infty$ $s\in\mr{X}$ $\mr{y}\clrest n\subseteq s$''.

In the first case, as $\PP$ satisfies the Laver-property, $\mr{x}$ cannot be a Cohen real over $V$, and hence a $q\leq p$ forces that $\mr{x}\in C$ for some nowhere dense closed set $C=[T]\in V$, $T\in \mrm{Nwd}\cap V$, and so $q\vd |\mr{X}\cap T|=\om$.

In the second case, we can shrink $\mr{X}$ and assume that it has an enumeration $\mr{X}=\{\mr{s}_k:k\in\om\}$ and there is a sequence $(\mr{n}_k)_{k\in\om}$ of $\PP$-names for an increasing sequence in $\om$ such that $p$ forces the following:
\[\tag{$\sharp$} \mr{n}_0=0,\;\mr{y}\clrest \mr{n}_k\subsetneq \mr{s}_k,\;\mr{y}\clrest (\mr{n}_k+1)\nsubseteq \mr{s}_k,\;\text{and}\;|\mr{s}_k|<\mr{n}_{k+1}\;\text{for every $k$}.\]
Now define the $\PP$-names $\mr{E}_m$ as follows: $p$ forces that if $m\in [\mr{n}_k,\mr{n}_{k+1})$ and $m>|\mr{s}_k|$ then $\mr{E}_m=\{\mr{y}\clrest m\}$, and if $m\in [\mr{n}_k,\mr{n}_{k+1})$ and $m\leq |\mr{s}_k|$ then $\mr{E}_m=\{\mr{y}\clrest m,\mr{s}_k\clrest m\}$. Now $\mr{E}_m\subseteq \,\!^m 2$ is of size $\leq 2$, hence applying the Laver property, there are a $q\leq p$ and a sequence $F_m\subseteq [\,\!^m 2]^{\leq 2}$ in $V$ such that $|F_m|= m+1$ and $q\vd \mr{E}_m\in F_m$ for every $m$, in particular, if $F'_m=\bigcup F_m\in [\,\!^m 2]^{\leq 2m+2}$  then $q\vd \mr{E}_m\subseteq F'_m$ for every $m$. Let $A=\{t\in \,\!^{<\om} 2:\forall$ $m\leq |t|$ $t\clrest m\in F'_m\}$. Then $A\in\mrm{Nwd}$ because for every $t\in \,\!^{<\om} 2$ there is a $t'\supseteq t$ such that $t'\notin F'_{|t'|}$ and hence no extension of $t'$ belongs to $A$, and of course $q\vd \mr{X}\subseteq A$.

We show that $\PP\ast\mbb{C}$ cannot $+$-destroy $\mrm{Nwd}$. First notice that if $\mr{X}$ is a $\mbb{C}$-name for a dense subset of $\,\!^{<\om} 2$, then there is a countable family $\{Y_n:n\in\om\}$ of dense subsets of $\,\!^{<\om}2$ such that if an $A\in\mrm{Nwd}$ has infinite intersection with all $Y_n$ (there is always such an $A$ because each $Y_n$ is dense) then $\vd_\mbb{C}|A\cap\mr{X}|=\om$. Why? Enumerate $\mbb{C}=\{q_n:n\in\om\}$ and define $Y_n=\{s\in \,\!^{<\om}2:\exists$ $q'\leq q_n$ $q'\vd s\in\mr{X}\}$. It is easy to see that this family satisfies our requirements. Now if $\mr{X}$ is a $\PP\ast\mbb{C}$-name and $(p,q)\vd\mr{X}\in\mrm{Nwd}^+$ then we can assume that $(p,q)$ forces that $\mr{X}$ is dense (because a condition below $(p,q)$ decides where $\mr{X}$ is dense and we can work inside that cone in $\,\!^{<\om}2$). Therefore there are $\PP$-names $\mr{Y}_n$ for dense subsets of $\,\!^{<\om}2$ such that $p$ forces the following: ``If $A\in\mrm{Nwd}$ and $|A\cap \mr{Y}_n|=\om$ for every $n$, then $q\vd_\mbb{C}|A\cap\mr{X}|=\om$''. Working in $V^\PP$, it is trivial to construct an antichain $\mr{Z}\subseteq \,\!^{<\om}2$ satisfying $(\sharp)$ which has infinite intersection with all $\mr{Y}_n$, and hence there is an $A\in\mrm{Nwd}\cap V$ covering $\mr{Z}$. It follows that $(p,q)\vd |A\cap \mr{X}|=\om$.
\end{proof}

\begin{rem}
Notice that the proof of part (2) of the last Proposition ``almost'' shows that destroying $\mrm{Nwd}$ requires Cohen reals: Assume that there is an $X\in [\,\!^{<\om} 2]^\om\cap V^\PP$ such that $|X\cap A|<\om$ for every $A\in\mrm{Nwd}\cap V$. Then either $[X]_\delta\ne\0$, i.e. $X$ contains an infinite chain $Y$ defining a real $y=\bigcup Y\in\,\!^\om 2$ or $X$ contains an infinite ``convergent'' antichain $Z$ defining $z\in\,\!^\om 2$ as the unique real such that $\forall$ $n$ $\forall^\infty$ $t\in Z$ $z\clrest n\subseteq t$. In the first case we can use the same argument as above but the second case in unclear.
\end{rem}

\begin{prob}
Does there exist a forcing notion $\PP$ which destroys $\mrm{Nwd}$ but does not add Cohen reals? (This problem might be quite difficult because we know that $\cov^*(\mrm{Nwd},\infty)=\cov(\mc{M})$ and hence iterated destruction of $\mrm{Nwd}$ implies adding Cohen reals. In other words, this problem resembles to the well-known ``half-a-Cohen-real'' problem, see \cite{zapletalhalf}.)
\end{prob}

\begin{prob}
Is there any reasonable characterisation of those tall Borel ideals $\mc{I}$ such that destruction of $\mc{I}$ implies $+$-destruction of it? (We will show later that there are $F_\sigma$ counterexamples too, e.g. $\mc{I}_{1/n}$.)
\end{prob}

We will discuss analytic P-ideals later.

\section{The $\mbb{M}(\mc{I}^*)$- and $\mbb{L}(\mc{I}^*)$-generic reals}\label{secmatlav}

In this section, applying Laflamme's filter games and his characterisations of the existence of winning strategies in these games, we will characterise when the generic reals added by the Mathias-Prikry forcing $\mbb{M}(\mc{I}^*)$ and the Laver-Prikry forcing $\mbb{L}(\mc{I}^*)$ (see below) $+$-destroy $\mc{I}$.

Fix an ideal $\mc{I}$ on $\om$. Then we can talk about infinite games of the following form (see \cite{Lafl1} and \cite{Lafl2}) $G(\mc{X},Y,\mc{O})$ where $\mc{X}=\mc{I}^*$ or $\mc{I}^+$, $Y=\om$ or $[\om]^{<\om}$, and $\mc{O}=\mc{I}^*$, $\mc{I}^+$, or $\mc{P}(\om)\setminus\mc{I}^*$. In the $n$th round Player \textbf{I} chooses an $X_n\in\mc{X}$ and Player \textbf{II} responds with a $k_n\in X_n$ (if $Y=\om$) or with an $F_n\in [X_n]^{<\om}$ (if $Y=[\om]^{<\om}$, respectively). Player \textbf{II} wins if $\{k_n:n\in\om\}\in\mc{O}$ (if $Y=\om$) or $\bigcup\{F_n:n\in\om\}\in\mc{O}$ (if $Y=[\om]^{<\om}$).

Let us show that Borel Determinacy (see \cite{martin}) implies that all these games are determined if $\mc{I}$ is Borel. First of all, we recall the setting of Borel Determinacy. Fix an infinite set $\Gamma$, a nonempty tree $T\subseteq\,\!^{<\om} \Gamma$ without terminal nodes (the set of all possible outcomes of the game), and a set $A\subseteq [T]=\{g\in\,\!^\om \Gamma:\forall$ $n$ $g\clrest n\in T\}$. The game $G(A,T)$ is played by players \textbf{I} and \textbf{II}, in the $n$th round \textbf{I} chooses an $x_n\in\Gamma$ and \textbf{II} responds with $y_n\in\Gamma$ such that $(x_0,y_0,x_1,\dots,x_n,y_n)\in T$. Player \textbf{I} wins if $(x_0,y_0,\dots,x_n,y_n,\dots)\in A$. Consider $\,\!^\om \Gamma$ as the power of a discrete space (in particular $[T]$ is a closed set). We know that if $A\subseteq [T]$ is Borel, then the game is determined, i.e. one of the players has a winning strategy.

Now assume that $\mc{I}$ is Borel and fix $\mc{X},Y$ and $\mc{O}$ as above. Let $\Gamma=\mc{P}(\om)$ and define $T$ accordingly, that is, $(X_0,F_0,X_1,F_1,\dots,X_{n-1},F_{n-1})\in T$ iff $X_k\in\mc{X}$ and $F_k\in [X_k]^1$ (if $Y=\om$) or $F_k\in [X_k]^{<\om}$ (if $Y=[\om]^{<\om}$) for every $k<n$, and let $A=\{g\in [T]:\bigcup\{g(2n+1):n\in\om\}\in\mc{P}(\om)\setminus \mc{O}\}$. Then a player has winning strategy in $G(A,T)$ iff he has winning strategy in $G(\mc{X},Y,\mc{O})$ (in other words, the two games are equivalent). To see that $A$ is Borel, notice that the map $\al:\,\!^\om\mc{P}(\om)\to\mc{P}(\om)$, $g\mapsto \bigcup\{g(2n+1):n\in\om\}$ is Borel because the preimage of $\{S\subseteq\om:m\in S\}$ is the open set $\{g\in\,\!^\om\mc{P}(\om):\exists$ $n$ $m\in g(2n+1)\}$. In particular, $A=\al^{-1}[\mc{P}(\om)\setminus \mc{O}]\cap [T]$ is Borel.

\smallskip
We say that a tall ideal $\mc{I}$ is a {\em weak P-ideal}, if every sequence $X_n\in\mc{I}^*$ ($n\in\om$) has an $\mc{I}$-positive pseudointersection, i.e. $\cov^*(\mc{I},+)>\om$. To define the next property, we need the following construction: For a fixed tall $\mc{I}$ on $\om$, we define $\mc{I}^{<\om}$ on $[\om]^{<\om}\setminus\{\0\}$ (see \cite{HrMi}) as the ideal generated by all sets of the form $A^{<\om}=\{x\in[\om]^{<\om}:A\cap x\ne\0\}$ for $A\in\mc{I}$ (notice that this family is closed for taking finite unions). We say that $\mc{I}$ is {\em $\om$-diagonalisable by $\mc{I}$-universal sets} if there is a countable family $\{X_n:n\in\om\}\subseteq(\mc{I}^{<\om})^+$ (i.e. $\forall$ $n$ $\forall$ $A\in\mc{I}$ $\exists$ $x\in X_n$ $A\cap x=\0$) such that \[\tag{$\ast$} \forall\;A\in\mc{I}\;\exists\;n\;\forall^\infty\;x\in X_n\;x\nsubseteq A.\]

\begin{thm} {\em (see \cite{Lafl1})}
In $G(\mc{I}^*,[\om]^{<\om},\mc{I}^+)$, \textbf{I} has a winning strategy iff $\mc{I}$ is not a weak P-ideal, and \textbf{II} has a  winning strategy iff $\mc{I}$ is $\om$-diagonalisable by $\mc{I}$-universal sets.
\end{thm}

\begin{cor}\label{mathias}
Let $\mc{I}$ be a tall Borel ideal on $\om$. Then the following are equivalent: (a) The $\mbb{M}(\mc{I}^*)$-generic $+$-destroys $\mc{I}$. (b) $\mbb{M}(\mc{I}^*)$ $+$-destroys $\mc{I}$. (c) There is a forcing notion which $+$-destroys $\mc{I}$. (d) $\cov^*(\mc{I},+)>\om$.
\end{cor}
\begin{proof} (a)$\rightarrow$(b)$\rightarrow$(c) is trivial and (c)$\rightarrow$(d) follows from Observations \ref{cardobs} (4). To show (d)$\to$(a), assume that $\cov^*(\mc{I},+)>\om$, i.e. that $\mc{I}$ is a weak P-ideal. As $G(\mc{I}^*,[\om]^{<\om},\mc{I}^+)$ is determined, \textbf{II} has winning strategy in this game, i.e. $\mc{I}$ $\om$-diagonalisable by $\mc{I}$-universal sets, fix such a family $\{X_n:n\in\om\}\subseteq (\mc{I}^{<\om})^+$. Notice that property ($\ast$) of this family is $\Ubf{\Pi}^1_1$ and hence holds in $V^{\mbb{M}(\mc{I}^*)}$ as well. It is straightforward to check that the sets
\[D_{n,m}=\big\{(s,F)\in\mbb{M}(\mc{I}^*):\exists\;x\in X_n\;\,x\subseteq s\setminus m\big\}\;\;(n,m\in\om)\]
are dense in $\mbb{M}(\mc{I}^*)$ and hence the generic $R\subseteq\om$ does not satisfy ($\ast$) (that is, $\forall$ $n$ $\exists^\infty$ $x\in X_n$ $x\subseteq R$), in particular, $V^{\mbb{M}(\mc{I}^*)}\models R\in\mc{I}^+$.
\end{proof}

If $\mc{I}$ is an ideal on $\om$, then the associated {\em Laver-Prikry forcing} $\mbb{L}(\mc{I}^*)$ is defined as follows (see \cite{BrHr} and \cite{HrMi}): $T\in\mbb{L}(\mc{I}^*)$ if $T\subseteq\,\!^{<\om}\om$ is a tree containing a (unique) $\mrm{stem}(T)\in T$ such that (i) $\forall$ $t\in T$ $(t\subseteq\mrm{stem}(T)$ or $\mrm{stem}(T)\subseteq t)$, and (ii)  $\mrm{ext}_T(t)=\{n:t^\frown(n)\in T\}\in\mc{I}^*$ for every $t\in T$, $\mrm{stem}(T)\subseteq t$; and $T_0\leq T_1$ if $T_0\subseteq T_1$.

$\mbb{L}(\mc{I}^*)$ is $\sigma$-centered (if $\mrm{stem}(T_0)=\mrm{stem}(T_1)$ then $T_0\| T_1$) and destroys $\mc{I}$: If $G$ is $\mbb{L}(\mc{I}^*)$-generic over $V$, $r_G=\bigcup\{\mrm{stem}(T):T\in G\}\in \,\!^\om\om$, and $Y_G=\mrm{ran}(r_G)$, then $Y_G\in [\om]^\om$ and $|Y_G\cap A|<\om$ for every $A\in\mc{I}^V$.

Perhaps the most important difference between $\mbb{M}(\mc{I}^*)$ and $\mbb{L}(\mc{I}^*)$ is that $\mbb{L}(\mc{I}^*)$ always adds dominating reals, and we know (see \cite{CRZ}) that for a Borel $\mc{I}$, $\mbb{M}(\mc{I}^*)$ adds dominating reals iff $\mc{I}$ is not $F_\sigma$. Another important, and for us relevant, difference between the two forcing notions is that while (see above) the $\mbb{M}(\mc{I}^*)$-generic object is $\mc{I}$-positive for every tall Borel $\mc{I}$ satisfying $\cov^*(\mc{I},+)>\om$, the $\mbb{L}(\mc{I}^*)$-generic object $Y_G$ is not necessarily, e.g. it is easy to see that $V^{\mbb{L}(\mc{ED}^*)}\models Y_{\mr{G}}\in\mc{ED}$. Of course, this does not mean that $\mbb{L}(\mc{ED}^*)$ cannot $+$-destroy $\mc{ED}$, and indeed, we already know that if $\PP$ destroys $\mc{ED}$ then it $+$-destroys $\mc{ED}$. We will see later that e.g. $\mbb{L}(\mc{Z}^*)$ cannot $+$-destroy $\mc{Z}$.

\smallskip
We say that an ideal $\mc{I}$ is {\em weakly Ramsey} if every $T\in\mbb{L}(\mc{I}^*)$ has a branch $x\in [T]$ such that $\mrm{ran}(x)\in\mc{I}^+$. We say that $\mc{I}$ is {\em $\om$-$+$-diagonalisable} if $\non^*(\mc{I},+)=\om$.

\begin{thm} {\em (see \cite{Lafl1})}
In $G(\mc{I}^*,\om,\mc{I}^+)$, \textbf{I} has a winning strategy iff $\mc{I}$ is not weakly Ramsey, and \textbf{II} has a winning strategy iff $\mc{I}$ is $\om$-$+$-diagonalisable.
\end{thm}

\begin{cor}
Let $\mc{I}$ be a tall Borel ideal on $\om$. Then the following are equivalent: (a) The $\mbb{L}(\mc{I}^*)$-generic $+$-destroys $\mc{I}$. (b) $\non^*(\mc{I},+)=\om$.
\end{cor}
\begin{proof}
(a)$\to$(b): First of all, (a) implies that $\mc{I}$ must be weakly Ramsey. If $T\in \mbb{L}(\mc{I}^*)$ does not have $\mc{I}$-positive branches, then this holds in $V^{\mbb{L}(\mc{I}^*)}$ as well because this property of $T$ is $\Ubf{\Pi}^1_1$, in particular, $T\vd Y_{\mr{G}}\in\mc{I}$. As $G(\mc{I}^*,\om,\mc{I}^+)$ is determined, $\non^*(\mc{I},+)=\om$.

(b)$\to$(a): If $\{X_n:n\in\om\}\subseteq\mc{I}^+$ witnesses $\non^*(\mc{I},+)=\om$, then this property of this family is $\Ubf{\Pi}^1_1$ hence it is still a witness of $\non^*(\mc{I},+)=\om$ in the extension as well. It is easy to show that the sets
\[D_{n,m}=\big\{T\in\mbb{L}(\mc{I}^*):\mrm{ran}(\mrm{stem}(T))\cap X_n\nsubseteq m\big\}\;\;(n,m\in\om)\]
are dense in $\mbb{L}(\mc{I}^*)$, in particular, in the extension $|Y_G\cap X_n|=\om$ for every $n$, and hence $Y_G\notin \mc{I}$.
\end{proof}

\begin{rem}
Let us recall the {\em Category Dichotomy} (see \cite{hrusakkatetov}): If \textbf{I} has winning strategy in $G((\mc{I}\clrest Y)^*,Y,(\mc{I}\clrest Y)^+)$ for some $Y\in\mc{I}^+$, then $\mc{ED}\leq_\mrm{K}\mc{I}\clrest X$ for some $X\in\mc{I}^+$; if \textbf{II} has winning strategy in $G((\mc{I}\clrest X)^*,X,(\mc{I}\clrest X)^+)$ for every $X\in\mc{I}^+$, then $\mc{I}\clrest X\leq_\mrm{K} \mrm{Nwd}$ for every $X\in\mc{I}^+$. In particular, if $\mc{I}$ is Borel, then one of these cases holds.

Now one may wonder if $\non^*(\mc{I},+)=\om$ has a characterisation using the Kat\v{e}tov preorder and the Category Dichotomy. This does not seem doable: We need that \textbf{II} has winning strategy in $G(\mc{I}^*,\om,\mc{I}^+)$, i.e. $\non^*(\mc{I},+)=\om$, but not necessarily in all $G((\mc{I}\clrest X)^*,X,(\mc{I}\clrest X)^+)$ games, e.g. if $\mc{I}=\mc{ED}\oplus\mrm{Nwd}$ (the ideal generated by disjoint copies of the two ideals, in this case on $(\om\times\om)\cup\mbb{Q}$) then $\non^*(\mc{I},+)=\om$ but $\mc{ED}=\mc{I}\clrest (\om\times\om)\nleq_\mrm{K}\mrm{Nwd}$.
\end{rem}

Notice that unlike in the case of $\mbb{M}(\mc{I}^*)$, the characterisation above says much less about $\mbb{L}(\mc{I}^*)$. In the next section, we will characterise when exactly $\mbb{L}(\mc{I}^*)$ $+$-destroys an analytic P-ideal $\mc{I}$.

\section{Fragile ideals}\label{secfrag}

\begin{df}
Let $\mc{I}$ be an ideal on $\om$. We say that $\mc{I}$ is {\em fragile} if there are a $Y\in\mc{I}^+$ and an $f:Y\to [\om]^{<\om}$ such that the following holds:
\begin{itemize}
\item[(a)] $f$ witnesses $\mc{I}^{<\om}\leq_\mrm{K}\mc{I}\clrest Y$, i.e. $f^{-1}[A^{<\om}]\in\mc{I}$ for every $A\in\mc{I}$;
\item[(b)] $\bigcup_{n\in H}f(n)\in\mc{I}^+$ for every infinite $H\subseteq Y$.
\end{itemize}
\end{df}

It is trivial to see that $\mc{I}^{<\om}\leq_\mrm{K} \mc{I}\upharpoonright Y$ for every $Y\in\mc{I}^+$, simply consider the map $n\mapsto\{n\}$. Loosely speaking, an ideal $\mc{I}$ is fragile if there is a very nontrivial reduction $\mc{I}^{<\om}\leq_\mrm{K} \mc{I}\upharpoonright Y$ for some $Y\in\mc{I}^+$. Notice that for Borel ideals, being fragile is a $\Ubf{\Sigma}^1_2$ property and hence absolute between $V$ and $V^\PP$.

\begin{fact}\label{fragile+dest} Let $\mc{I}$ be a fragile Borel ideal witnessed by $f:Y\to [\om]^{<\om}$. If a forcing notion $\PP$ destroys $\mc{I}\clrest Y$, then it $+$-destroys $\mc{I}$. In particular, if $\mc{I}\clrest Y\leq_\mrm{K}\mc{I}$ (e.g. $Y=\om$ or $\mc{I}$ is $K$-uniform, that is, $\mc{I}\clrest Y\leq_\mrm{K}\mc{I}$ for every $Y\in\mc{I}^+$), then destroying $\mc{I}$ implies $+$-destroying it.
\end{fact}
\begin{proof}
Let $\mathring{H}$ be a $\PP$-name such that $\vd_\PP$``$\mathring{H}\in [Y]^\om$ and $|\mathring{H}\cap A|< \om$ for every $A\in\mc{I}\cap V$''. In the extension,  $\mathring{X}=\bigcup_{n\in \mathring{H}}f(n)\in\mc{I}^+$ because (b) is a $\Ubf{\Pi}^1_1$ property of $(f(n))_{n\in Y}\in\,\!^Y\mrm{Fin}$; and $|\mathring{X}\cap A|<\om$ for every $A\in\mc{I}\cap V$ because for such an $A$, $f^{-1}[A^{<\om}]\in\mc{I}\cap V$, therefore $\mr{H}\cap f^{-1}[A^{<\om}]=\{n\in\mr{H}:f(n)\cap A\ne\0\}$ is finite.
\end{proof}

One can easily show that $\mc{ED}$ and $\mrm{Conv}$ are not fragile. Also, it is easy to see that $\mrm{Fin}\otimes\mrm{Fin}$ and $\mrm{Nwd}$ are K-uniform, and we know that $\mrm{Fin}\otimes\mrm{Fin}$ cannot be $+$-destroyed, and that $\mrm{Nwd}$ can be destroyed without being $+$-destroyed, hence they are not fragile either.  Our flagship example of a fragile K-uniform ideal is $\mc{ED}_\mrm{fin}$ (see the proof of Proposition \ref{edfin} (1)). Let us present ``very'' fragile summable and density ideals as well:

\begin{exa}
There are a tall summable ideal $\mc{I}_h$ and a tall density ideal $\mc{Z}_{\vec\mu}$ which are fragile with $Y=\om$ in the definition (and hence destruction of these ideals implies $+$-destruction of them): Fix an interval partition $(P_n)$ such that $|P_{n+1}|=2^{n+1}|P_n|$, let $h(k)=\mu_n(\{k\})=2^{-n}$ if $k\in P_n=\mrm{supp}(\mu_n)$, fix partitions $P_{n+1}=\bigcup_{k\in P_n}P^{n+1}_k$ such that $|P^{n+1}_k|=2^{n+1}$ for every $k\in P_n$ and $n\in\om$, and define $f(k)=P^{n+1}_k$ if $k\in P_n$.
\end{exa}

We will show that fragility plays a fundamental role when discussing whether an analytic P-ideal $\mc{I}$ is $+$-destroyed by $\mbb{L}(\mc{I}^*)$ or not, but first let us show some less trivial examples of not fragile ideals.

\begin{obs}\label{nicerest}
If an analytic P-ideal $\mc{I}=\mrm{Exh}(\varphi)$ is fragile witnessed by $f:Y\to [\om]^{<\om}$, then there are a $Z\subseteq Y$, $Z\in\mc{I}^+$ and an $\eps>0$ such that with $g=f\clrest Z$ (clearly, $g$ also witnesses fragility of $\mc{I}$) the following holds: (i) $g^{-1}[k^{<\om}]=\{z\in Z:g(z)\cap k\ne\0\}$ is finite for every $k\in\om$, and (ii) $\varphi(g(z))>\eps$ for every $z$. Why? Fix a $C\in\mc{I}$ such that $g^{-1}[k^{<\om}]\subseteq^* C$ for every $k$, and define $Z=Y\setminus C$. Then (i) holds. To show that (ii) also holds, assume on the contrary that there is a sequence $z_0<z_1<\cdots$ in $Z$ such that $\varphi(g(z_i))\xrightarrow{i\to\infty} 0$. Then there is an infinite $H=\{i_0<i_1<\cdots\}\subseteq\om$ such that $g(z_{i_0})<g(z_{i_1})<\cdots$ (because of (i)) and $\varphi(g(z_{i_m}))<2^{-m}$. Now $\bigcup_{i\in H} g(z_i)\in\mc{I}$ (because $\varphi$ is $\sigma$-subadditive) but this contradicts (b) in the definition of fragility.
\end{obs}

\begin{prop}
$\mc{I}_{1/n}$ is not fragile.
\end{prop}
\begin{proof}
We will work with the canonical isomorphic copy $\mc{I}=\{A\subseteq\,\!^{<\om}2:h(A):=\sum_{s\in A}2^{-|s|}<\infty\}$ of $\mc{I}_{1/n}$, and assume on the contrary $\mc{I}$ is fragile. Applying the last Observation, we can assume that there are a $Y\in\mc{I}^+$, an $f:Y\to [\,\!^{<\om}2]^{<\om}$, and an $\eps>0$ such that $f^{-1}[A^{<\om}]\in\mc{I}$ for every $A\in\mc{I}$, $f^{-1}[(\,\!^{<n}2)^{<\om}]$ is finite for every $n$, and $h(f(y))>\eps$ for every $y\in Y$ (this implies that $f$ witnesses fragility of $\mc{I}$).

Further restricting $f$, we can assume that there is a sequence $m_0=0<m_1<\cdots$ such that if $I_n:=[m_n,m_{n+1})$, $B_n:=\bigcup_{k\in I_n}\,\!^k 2$, and $Y_n=Y\cap B_n$, then $h(Y_n)=1$ and $\cup f[Y_n]\subseteq B_n$.

Now, let $Z_n=\,\!^{m_{n+1}}2$, $Z=\bigcup_{n\in\om}Z_n\in\mc{I}^+$, fix a partition $Z_n=\bigcup_{y\in Y_n}Z^n_y$ such that $h(Z^n_y)=h(y)$,   and define $g:Z\to [\,\!^{<\om}2]^{<\om}$ by $g(z)=f(y)$ if $z\in Z^n_y$. It is trivial to show that $g$ still witnesses fragility of $\mc{I}$.

There is an $N$ such that $2^{-m_N}<\eps$, and hence, by shrinking values of $g$, we can assume that $\eps<h(g(z))<2\eps$ for every $z\in \bigcup_{n\geq N}Z_n$. Fix an $n\geq N$. We claim that we can pick single points at each level in $B_n$, that is, we can construct an $F_n=\{t_k:k\in I_n\}$ where $t_k\in \,\!^k2$ such that $h(g^{-1}[F_n^{<\om}])\geq \eps/(1+2\eps)$. Then we are done, because $h(F_n)<2^{-m_n+1}$ and hence $A=\bigcup_{n\geq N}F_n\in\mc{I}$ but $h(g^{-1}[A^{<\om}])=\infty$, a contradiction.

For every $z\in Z_n$ define the ``vector'' $v_z\in \,\!^{I_n}[0,\infty)$, $v_z(k)=h(g(z)\cap \,\!^k2)=|g(z)\cap \,\!^k 2|\cdot 2^{-k}$. Now $\eps<\sum v_z=h(g(z))< 2\eps$ for every $z\in Z_n$. Picture these vectors as columns next to each other in a $I_n\times Z_n$ matrix. The next diagram may help following the construction of $F_n$ below:
{\small
\[\def\arraystretch{1.2}
\begin{array}{>{\centering\arraybackslash$} p{1.1cm} <{$}||>{\centering\arraybackslash$} p{1.1cm} <{$}|>{\centering\arraybackslash$} p{1.1cm} <{$}|>{\centering\arraybackslash$} p{1.1cm} <{$}|>{\centering\arraybackslash$} p{1.1cm} <{$}||>{\raggedright\arraybackslash$} p{5.8cm} <{$}}
& Z_{n,m_n} & Z_{n,m_n+1} & Z_{n,m_n+2} & \hdots & \\ \cline{1-6}
m_n & &  &   & \cdots & \leadsto\lceil\sum\text{``whole row''}\rceil=a_{m_n} \\\cline{2-6}
m_n+1 & \cellcolor{gray} &  & & \cdots &  \leadsto\lceil\sum\text{``row - gray section''}\rceil=a_{m_n+1} \\ \cline{3-6}
m_n+2 & \cellcolor{gray} & \cellcolor{gray} & & \cdots & \leadsto\lceil\sum\text{``row - gray section''}\rceil=a_{m_n+2} \\ \cline{4-6}
m_n+3 & \cellcolor{gray} & \cellcolor{gray} &\cellcolor{gray} & \cdots & \leadsto\lceil\sum\text{``row - gray section''}\rceil=a_{m_n+3}\\ \cline{5-6}
\vdots & \cellcolor{gray}\vdots & \cellcolor{gray}\vdots & \cellcolor{gray}\vdots & \ddots & \multicolumn{1}{c}{\vdots}\\
\end{array}\]}

We will construct $t_k$ by recursion on $k\in I_n$ as follows: Let $a_{m_n}=\lceil\sum_{z\in Z_n}v_z(m_n)\rceil$. A trivial version of Fubini's theorem shows that there is a point $t_{m_n}\in \,\!^{m_n}2$ such that $|\{z\in Z_n:t_{m_n}\in g(z)\}|=|g^{-1}[\{t_{m_n}\}^{<\om}]|\geq a_{m_n}$, fix a $Z_{n,m_n}\in [Z_n]^{a_{m_n}}$ such that $t_{m_n}\in g(z)$ for every $z\in Z_{n,m_n}$. If we are done below $k$, define
\[ a_k=\big\lceil \sum\big\{v_z(k):z\in Z_n\setminus \big(Z_{n,m_n}\cup\cdots \cup Z_{n,k-1}\big)\big\}\big\rceil,\] fix a $t_k\in \,\!^k2$ such that $|g^{-1}[\{t_k\}^{<\om}]|\geq a_k$ and a $Z_{n,k}\in [Z_n\setminus (Z_{n,m_n}\cup\cdots \cup Z_{n,k-1})]^{a_k}$ such that $t_k\in g(z)$ for every $z\in Z_{n,k}$.

Of course, it is possible that $Z_n\setminus (Z_{n,m_n}\cup\cdots \cup Z_{n,k-1})=\0$. Then declare the empty sum to be $0$, let $t_k\in \,\!^k 2$ be arbitrary and $Z_{n,k}=\0$. Now the sum of all elements in this matrix is $S=\sum_{k\in I_n}\sum_{z\in Z_n}v_z(k)>|Z_n|\eps=2^{m_{n+1}}\eps$, also $S\leq \sum_{k\in I_n}a_k+\sum$``grey section'', and, by extending the grey section with all elements above it, we obtain that
\[ \sum\text{``grey section''}\leq \sum_{k\in I_n}\sum_{z\in Z_{n,k}}\sum_{\ell\in I_n} v_z(\ell)<\sum_{k\in I_n}a_k 2\eps.\]
Therefore, $2^{m_{n+1}}\eps<(1+2\eps)\sum_{k\in I_n}a_k$, and so $h(g^{-1}[F_n^{<\om}])\geq 2^{-m_{n+1}}\sum_{k\in I_n}a_k>\eps/(1+2\eps)$.
\end{proof}

\begin{prop}
$\mc{Z}$ is not fragile.
\end{prop}
\begin{proof}
Let $P_n=[2^n,2^{n+1})$, and $\varphi(A)=\sup_{n\to\infty}|A\cap P_n|/2^n$. Then clearly $\|A\|_\varphi=\limsup_{n\to\infty}|A\cap P_n|/2^n$ and $\mc{Z}=\mrm{Exh}(\varphi)$. Assume on the contrary that $\mc{Z}$ is fragile, that is, by applying Observation \ref{nicerest}, that there are a $Y\in\mc{Z}^+$, an $f:Y\to [\om]^{<\om}$, and an $\eps>0$ such that $f^{-1}[A^{<\om}]\in\mc{Z}$ for every $A\in\mc{Z}$, $f^{-1}[n^{<\om}]$ is finite for every $n$, and $\varphi(f(y))>\eps$ for every $y\in Y$. Furthermore, we can assume the following:
\begin{itemize}
\item[(i)] $f(y)\subseteq P_{k(y)}$ for some $k(y)\in\om$ for every $y\in Y$;
\item[(ii)] there is a sequence $m(0)<m(1)<m(2)<\cdots$ such that $Y=\bigcup_{n\in\om} Y_n$ where $Y_n\subseteq P_{m(n)}$ and $\varphi(Y_n)=|Y_n|/2^{m(n)}>\|Y\|_\varphi-2^{-n}$ for every $n$;
\item[(iii)] $\max\{k(y):y\in Y_n\}<\min\{k(y):y\in Y_{n+1}\}$ for every $n$.
\end{itemize}
We will define a sequence $A_0<A_1<\cdots$ of finite sets such that $A_n\subseteq \bigcup_{y\in Y_n}P_{k(y)}$, $|A_n\cap P_k|\leq 1$ for every $k$, and $\varphi(f^{-1}[A_n^{<\om}])=|f^{-1}[A_n^{<\om}]|/2^{m(n)}\geq \eps \varphi(Y_n)$. In particular, if $A=\bigcup_{n\in\om}A_n$ then $A\in\mc{Z}$ and $\|f_n^{-1}[A^{<\om}]\|_\varphi\geq\eps\|Y\|_\varphi$, a contradiction.

Fix an $n$, let $\{k_i:i<d\}=\{k(y):y\in Y_n\}$ be an enumeration, and partition $Y_n$ accordingly, that is, $Y_n=\bigcup_{i<d} Q_i$ where $y\in Q_i$ iff $k(y)=k_i$. Now in $P_{k_i}$ we have $|Q_i|$ many sets $f(y)\in [P_{k_i}]^{>\eps|P_{k_i}|}$. Counting with multiplicity, at least $|Q_i|\eps|P_{k_i}|$ many points are covered by these sets in $P_{k_i}$, and hence there must be an $a_{n,i}\in P_{k_i}$ which is contained at least in $\eps|Q_i|$ many of these sets. Now if $A_n=\{a_{n,i}:i<d\}$ then
\[ |f^{-1}[A_n^{<\om}]|=\big|\big\{y\in Y_n:A_n\cap f(y)\ne\0\big\}\big|\geq \sum_{i<d}\eps|Q_i|=\eps|Y_n|,\]
and so $\varphi(f^{-1}[A_n^{<\om}])\geq\eps\varphi(Y_n)$.
\end{proof}

In the case of $\mrm{tr}(\mc{N})$, we know more, as one can show that $\mrm{tr}(\mc{N})$ is K-uniform (or see e.g. \cite[Thm. 2.1.17]{meza}), Fact \ref{fragile+dest} and the following easy one imply that $\mrm{tr}(\mc{N})$ is not fragile either.

\begin{fact}
The random forcing cannot $+$-destroy $\mrm{tr}(\mc{N})$, in particular,  $\cov^*(\mrm{tr}(\mc{N}),+)<\cov^*(\mrm{tr}(\mc{N}),\infty)$ in the random model.
\end{fact}
\begin{proof}
We know (see e.g. \cite[Lem. 6.3.12]{BaJu}) that $V^\mbb{B}\models \lam^*(\,\!^\om 2\cap V)>0$ and hence $V^\mbb{B}\models \lam^*(\,\!^\om 2\cap V)=1$, i.e. every positive Borel set coded in $V^\mbb{B}$, e.g. $[X]_\delta$ for $X\in\mrm{tr}(\mc{N})^+\cap V^\mbb{B}$, contains ground model reals. In the case of $[X]_\delta$, the branch associated to such a ground model real has infinite intersection with $X$.
\end{proof}

In the proof of the main result of this section, we will need the following technical lemma:

\begin{lem}\label{fragilepropK}
An analytic P-ideal $\mc{I}$ is not fragile if, and only if the following holds for a (equivalently, for every) lsc submeasure $\varphi$ generating $\mc{I}$: \textbf{IF} $\eps\in (0,\|\om\|_\varphi)$, $Y_n\in\mc{I}^+$, and $f_n:Y_n\to\mc{H}_{\varphi,\eps}:=\{F\in [\om]^{<\om}:\varphi(F)>\eps\}$ such that $f_n^{-1}[\{H\}]\in\mc{I}$ for every $n\in\om$ and $H\in [\om]^{<\om}$, \textbf{THEN} there is an $A\in\mc{I}$ such that $f_n^{-1}[A^{<\om}]\in\mc{I}^+$ for every $n$.
\end{lem}
\begin{proof}
Assume first that $\mc{I}$ is fragile witnessed by an $f:Y\to [\om]^{<\om}$. By Observation \ref{nicerest}, we can assume that $\ran(f)\subseteq\mc{H}_{\varphi,\eps}$ for some $\eps>0$, therefore the trivial sequence $f_n=f$ witnesses that the second statement fails.

\smallskip
Conversely, assume that $\mc{I}=\mrm{Exh}(\varphi)$ and that the second statement does not hold, that is, there are $Y_n\in\mc{I}^+$ and $f_n:Y_n\to\mc{H}_{\varphi,\eps}$ for some $\eps>0$ such that $f^{-1}_n[\{H\}]\in\mc{I}$ for every $n$ and $H$, and $\forall$ $A\in\mc{I}$ $\exists$ $n_A\in\om$ $f^{-1}_{n_A}[A^{<\om}]\in\mc{I}$. We can assume that $\mc{I}$ is tall, otherwise it is trivially fragile. By shrinking the values of these functions, we can assume that $\varphi(f_n(y))\geq \eps$ for every $n$ and $y\in Y_n$ but $\varphi(F)<\eps$ for every $F\subsetneq f_n(y)$ (this can be show by recursively removing points from $f_n(y)$ until we can).  Let $\mc{A}_n=\{A\in\mc{I}:n_A=n\}$. Then $\mc{I}=\bigcup_{n\in\om}\mc{A}_n$ and hence there is an $N$ such that $\mc{A}_N$ is $\subseteq^*$-cofinal in $\mc{I}$. Then
\[\tag{$\star$} \forall\;B\in \mc{I}\;\exists\;m\in\om\;f^{-1}_N[(B\setminus m)^{<\om}]\in\mc{I}\]
holds, in particular, the set $\mrm{Bad}=\{k\in\om:f^{-1}_N[\{k\}^{<\om}]\in\mc{I}^+\}$ is finite (otherwise there was an infinite $B\in\mc{I}$ such that $f_N^{-1}[\{k\}^{<\om}]\in\mc{I}^+$ for every $k\in B$).

Now, fix a $C\in\mc{I}$ which almost contains  $f_N^{-1}[\{k\}^{<\om}]$ for every $k\in\om\setminus\mrm{Bad}$, and define $Y=Y_N\setminus (C\cup f_N^{-1}[\mc{P}(\mrm{Bad})])\in\mc{I}^+$ and $f:Y\to[\om]^{<\om}$, $f(y)=f_N(y)\setminus\mrm{Bad}$. Then $f$ witnesses $\mc{I}^{<\om}\leq_\mrm{K}\mc{I}\clrest Y$ because if $B\in\mc{I}$ and $m$ is as in ($\star$) then  $f^{-1}[B^{<\om}]\subseteq f^{-1}[m^{<\om}]\cup f^{-1}[(B\setminus m)^{<\om}]$ where $f^{-1}[m^{<\om}]\in\mc{I}$ because $f(y)\cap \mrm{Bad}=\0$ for every $y\in Y$, and $f^{-1}[(B\setminus m)^{<\om}]\in\mc{I}$ because of ($\star$).

It is left to show that $\bigcup_{n\in H}f(n)\in\mc{I}^+$ for every $H\in [Y]^\om$ (then $f$ witnesses that $\mc{I}$ is fragile). By removing $C$ from $Y_N$, we ensured that $f^{-1}[\{k\}^{<\om}]$ and hence $f^{-1}[k^{<\om}]$ is finite for every $k$. In particular, if $H=\{a_0<a_1<\cdots\}\subseteq Y$ is infinite, then we can assume that $f(a_0)<f(a_1)<\cdots$. To finish the proof we show that there is an $\eps'>0$ such that $\varphi(f(y))\geq\eps'$ for every $y\in Y$ (and hence $\|\bigcup_{i\in\om}f(a_i)\|_\varphi\geq\eps'$). We know that $f_N(y)\nsubseteq \mrm{Bad}$ for any $y\in Y$, and hence $\varphi(f_N(y)\cap \mrm{Bad})<\eps$. This implies that  $\eps'=\eps-\max\{\varphi(F):F\subseteq\mrm{Bad},\varphi(F)<\eps\}$ is as desired.
\end{proof}

\begin{thm}
Let $\mc{I}$ be an analytic P-ideal. Then $\mbb{L}(\mc{I}^*)$ cannot $+$-destroy $\mc{I}$ iff $\mc{I}$ is not fragile.
\end{thm}
\begin{proof}
A trivial density argument shows that $\mbb{L}(\mc{I}^*)$ destroys $\mc{I}\clrest Y$ for every $Y\in\mc{I}^+\cap V$, and hence applying Fact \ref{fragile+dest}, if $\mc{I}$ is fragile, then $\mbb{L}(\mc{I}^*)$ $+$-destroys $\mc{I}$.

Conversely, assume that $\mc{I}$ is not fragile. Let $\varphi$ be an lsc submeasure such that $\mc{I}=\mrm{Exh}(\varphi)$, and let $\mr{X}$ be an $\mbb{L}(\mc{I}^*)$-name for a $\mc{I}$-positive set, fix a $T_0\in\mbb{L}(\mc{I}^*)$ and an $\eps>0$ such that $T_0\vd\|\mr{X}\|_\varphi>\eps$. We show that there is an $A\in\mc{I}$ such that $T_0\vd|\mr{X}\cap A|=\om$.

Fix a bijection $e:[\om]^{<\om}\to \om$ and a sequence $(\mr{H}_m)_{m\in\om}$ of $\mbb{L}(\mc{I}^*)$-names  such that $T_0$ forces the following for every $m$: (i) $\mr{H}_m\subseteq\mr{X}$ and $\varphi(\mr{H}_m)>\eps$, (ii) $\max(\mr{H}_m)<\min(\mr{H}_{m+1})$, and (iii) $e(\mr{H}_m)>\mr{\ell}(m)$ where $\mr{\ell}$ is an $\mbb{L}(\mc{I}^*)$-name for the generic $\om\to\om$ function.

We will use a rank argument on $Q=T_0\cap\mrm{stem}(T_0)^\uparrow$. We say that an $s\in Q$ {\em favors} $\mr{H}_m=E$ (for some $E\in\mc{H}_{\varphi,\eps}$) if
\[ \forall\;T\leq T_0\;\big(\mrm{stem}(T)=s\longrightarrow T\nVdash \mr{H}_m\ne E\big).\]
Now define the rank functions $\varrho_m$ on $Q$ for every $m\in\om$ by recursion as follows: $\varrho_m(s)=0$ if there is an $E^s_m\in\mc{H}_{\varphi,\eps}$ such that $s$ favors $\mr{H}_m=E^s_m$; and $\varrho_m(s)=\al>0$ if $\varrho_m(s)\not<\al$ and $\{n:\varrho_m(s^\frown(n))<\al\}\in\mc{I}^+$. It is trivial to show that $\dom(\varrho_m)=Q$.

We claim that $\varrho_m(s)>0$ whenever $m\geq|s|$. Fix conditions $S_k\leq T_0$ for every $k$ such that $\mrm{stem}(S_k)=s$ and $\mrm{ext}_{S_k}(t)\subseteq\om\setminus k$ for every $t\in S_k\cap s^\uparrow$. Now if $m\geq |s|$ then $S_k\vd k\leq \mr{\ell}(m)< e(\mr{H}_m)$, and hence $s$ cannot favor $\mr{H}_m=E$ for any $E$ because $e(E)=k$ for some $k$.

If $\varrho_m(s)=1$ then $Y_{m,s}=\{n:\varrho_m(s^\frown (n))=0\}\in\mc{I}^+$, and $s^\frown(n)$ favors $\mr{H}_m=E^{s^\frown (n)}_m$ for every $n\in Y_{m,s}$. Define $f_{m,s}:Y_{m,s}\to\mc{H}_{\varphi,\eps}$ as $f_{m,s}(n)=E^{s^\frown (n)}_m$. Notice that $f_{m,s}^{-1}[\{E\}]\in\mc{I}$ for every $E$ because otherwise $s$ would favor $\mr{H}_m=E$, and hence $\varrho_m(s)$ would be $0$.

Applying Lemma \ref{fragilepropK}, there is an $A\in\mc{I}$ such that $f_{m,s}^{-1}[A^{<\om}]\in \mc{I}^+$ whenever $\varrho_m(s)=1$. We claim that $T_0\vd|\mr{X}\cap A|=\om$. Otherwise, there is a $T\leq T_0$ with stem $s$ forcing $\mr{X}\cap A\subseteq M$ for some $M\in\om$. Fix an $m\geq M,|s|$, then $\varrho_m(s)>0$ and hence there is a $t\in T$ above its stem of $m$-rank $1$ (this can be shown by induction on $\varrho_m(s)$). As $f^{-1}_{m,t}[A^{<\om}]\in\mc{I}^+$, we know that there is an $n\in\mrm{ext}_T(t)\cap f^{-1}_{m,t}[A^{<\om}]$, in particular,  $t^\frown(n)$ favors $\mr{H}_m=f_{m,t}(n)=E^{t^\frown (n)}_m\subseteq\om\setminus m\subseteq \om\setminus M$ and this set has nonempty intersection with $A$. If $T'\leq T\clrest (t^\frown(n))$ forces $\mr{H}_m=E^{t^\frown(n)}_m$ then $T'\vd \mr{X}\cap A\nsubseteq M$, a contradiction.
\end{proof}

Unfortunately, it is still unclear what happens under iterations:

\begin{prob} Let $\mc{I}$ be an analytic P-ideal which is not fragile. Is it true that finite support iterations of $\mbb{L}(\mc{I}^*)$ cannot $+$-destroy $\mc{I}$? Or at least, is $\cov^*(\mc{I},+)<\cov^*(\mc{I},\infty)$ consistent? (We have seen that for $\mc{I}=\mrm{tr}(\mc{N})$, this strict inequality holds in the random model.) Similarly, one can ask about possible separations of the uniformity numbers.
\end{prob}

\section{Remarks on $*$-destruction}\label{secstardest}
Let us begin with a short introduction to (Borel) Tukey connections using Fremlin's notations (for more details, see \cite{Fr} or \cite{Bl}): A triple $\mc{R}=(A,R,B)$ is a {\em (supported) relation} if $R\subseteq A\times B$, $A=\mrm{dom}(R)$, $B=\mrm{ran}(R)$, and $\nexists$ $b\in B$ $\forall$ $a\in A$ $aRb$ (where $aRb$ stands for $(a,b)\in R$). The relation $\mc{R}$ is {\em Borel} if $A,B\subseteq\,\!^\om\om$ and $R$ are Borel sets. For a given $\mc{R}=(A,R,B)$, a set $X\subseteq A$ is {\em $\mc{R}$-unbounded} if there is no $b\in B$ $R$-above every $a\in X$, and a $Y\subseteq B$ is {\em $\mc{R}$-cofinal} if for every $a\in A$ there is a $b\in Y$ such that $aRb$. We define the {\em unbounding} and {\em dominating numbers} of $\mc{R}$ as follows:
\begin{align*}
\mf{b}(\mc{R})&=\min\big\{|X|:X\subseteq A\;\text{is}\;\mc{R}\text{-unbounded}\big\}\\
\mf{d}(\mc{R}) &=\min\big\{|Y|:Y\subseteq B\;\text{is}\;\mc{R}\text{-cofinal}\big\}
\end{align*}
Every cardinal invariant from Cicho\'n's diagram (see \cite{Bl}) and from above can easily be written of this form, for example $\cov(\mc{M})=\mf{d}(\,\!^\om 2,\in,\mc{M})$, $\mf{d}=\mf{d}(\,\!^\om\om,\leq^*,\,\!^\om\om)$, $\add(\mc{N})=\mf{b}(\mc{N},\subseteq,\mc{N})$, $\non^*(\mc{I},\infty)=\mf{b}([\om]^\om,\in,\wh{\mc{I}})=\mf{b}([\om]^\om,R_\mrm{ii},\mc{I})$, and $\cov^*(\mc{I},+)=\mf{d}(\mc{I}^+,\in,\wh{\mc{I}})=\mf{d}(\mc{I}^+,R_\mrm{ii},\mc{I})$ where $XR_\mrm{ii} A$ iff $|X\cap A|=\om$, etc. Notice that each unbounding number is actually a dominating number and vice versa: If $\mc{R}^\perp=(B,\neg R^{-1},A)$ then $\mf{b}(\mc{R}^\perp)=\mf{d}(\mc{R})$ and $\mf{d}(\mc{R}^\perp)=\mf{b}(\mc{R})$. Furthermore, all these underlying relations can be seen as Borel (in the cases of $\mc{M}$, $\mc{N}$, and $\wh{\mc{I}}$ we can use natural codings of nice bases of these ideals).

\smallskip
Fremlin and Vojt\'a\v{s} isolated a method of comparing cardinal invariants of these forms (see \cite{frcichon} and \cite{Vojtas}), it turned out that most of the known inequalities can be proved by this method, and most importantly, applying this approach we immediately obtain more than ``just'' inequalities between cardinal invariants. For given (Borel) $\mc{R}_0=(A_0,R_0,B_0)$ and $\mc{R}_1=(A_1,R_1,B_1)$, we say that $\mc{R}_0$ is {\em (Borel) Tukey-reducible to $\mc{R}_1$}, $\mc{R}_0\leq_\mrm{(B)T} \mc{R}_1$,\footnote{Some authors, including of \cite{Bl} and \cite{tukeymor}, would write $\mc{R}_1\leq_\mrm{T}\mc{R}_0$ here.}  if there are (Borel) maps $\al:A_0\to A_1$ and $\be:B_1\to B_0$ such that $aR_0\be(b)$ whenever $\al(a)R_1 b$, in a diagram:
\begin{diagram}
\be(b)\in & B_0 & & \lTo^\be && B_1& \ni b\\
R_0&  & {\color{white} n} & \lImplies &  {\color{white} n}& & R_1\\
a\in & A_0 && \rTo_\al & & A_1 & \ni\al(a)
\end{diagram}
We use the equivalence notation $\equiv_\mrm{(B)T}$ for bireducibilities.

\smallskip
Now if $\mc{R}_0\leq_\mrm{T}\mc{R}_1$ then clearly
\[\tag{Ia,Ib} \mf{b}(\mc{R}_0)\geq\mf{b}(\mc{R}_1)\;\,\text{and}\,\;\mf{d}(\mc{R}_0)\leq\mf{d}(\mc{R}_1).\] If $\mc{R}_0\leq_\mrm{BT}\mc{R}_1$ then we know more but first we need the following definitions: Let $\PP$ be a forcing notion and $\mc{R}=(A,R,B)$ be a Borel relation. We say that $\PP$ is {\em $\mc{R}$-bounding} if $\vd_\PP\forall$ $a\in A\cap V[\mr{G}]$ $\exists$ $b\in B\cap V$ $aRb$, i.e. $B\cap V$ remains $\mc{R}$-cofinal in $V^\PP$; and we say that $\PP$ is {\em $\mc{R}$-dominating} if $\vd_\PP\exists$ $b\in B\cap V[\mr{G}]$ $\forall$ $a\in A\cap V$ $aRb$, i.e. $A\cap V$ is $\mc{R}$-bounded in $V^\PP$. For example, $\PP$ adds Cohen reals iff it is $(\mc{M},\not\ni,\,\!^\om 2)$-dominating, $\PP$ is $\,\!^\om\om$-bounding iff it is $(\,\!^\om\om,\leq^*,\,\!^\om\om)$-bounding, $\vd_\PP\,\!^\om 2\cap V\notin \mc{N}$ iff $\PP$ is $(\mc{N},\not\ni,\,\!^\om 2)$-bounding\footnote{Notice that, if $\PP$ is $(\mc{N},\not\ni,\,\!^\om 2)$-bounding, then it is not $(\,\!^\om 2,\in,\mc{N})$-dominating, but the reverse implication requires that $\PP$ satisfies some sort of homogeneity.}, $\PP$ $+$-destroys $\mc{I}$ iff it is $(\mc{I},\neg R_\mrm{ii},\mc{I}^+)$-dominating, and $\PP$ $*$-destroys $\mc{I}$ iff it is $(\mc{I},\neg R_\mrm{ii},\mc{I}^*)$-dominating iff it is $(\mc{I}\subseteq^*,\mc{I})$-dominating, etc.

It is easy to see that if $\mc{R}_0\leq_\mrm{BT}\mc{R}_1$ then
\begin{align*}
\tag{IIa} \PP\;\text{is}\;\mc{R}_1\text{-bounding}& \longrightarrow \PP\;\text{is}\; \mc{R}_0\text{-bounding},\\
\tag{IIb} \PP\;\text{is}\;\mc{R}_1\text{-dominating}& \longrightarrow \PP\;\text{is}\;\mc{R}_0\text{-dominating}.
\end{align*}

For example, now we can add the ``missing'' last point to Observations \ref{cardobs}:
\begin{obs}
Let $\mc{I}$ and $\mc{J}$ be Borel ideals and assume that  $(\mc{I},\subseteq^*,\mc{I})\leq_\mrm{BT}(\mc{J},\subseteq^*,\mc{J})$, that is, there are Borel functions $\al:\mc{I}\to\mc{J}$ and $\be:\mc{J}\to\mc{I}$ such that for every $A\in \mc{I}$ and $B\in\mc{J}$, $\al(A)\subseteq^* B$ implies $A\subseteq^* \be(B)$. Then  $\cof^*(\mc{I},\infty)=\non^*(\mc{I},*)\leq\non^*(\mc{J},*)=\cof^*(\mc{J},\infty)$ and $\add^*(\mc{I},\infty)=\cov^*(\mc{I},*)\geq\cov^*(\mc{J},*)=\add^*(\mc{J},\infty)$; and if $\PP$ cannot $*$-destroy $\mc{I}$ then $\PP$ cannot $*$-destroy $\mc{J}$ either, and dually, if $\vd_\PP\mc{I}^*\cap V\in\wh{\mc{I}}$ then $\vd_\PP\mc{J}^*\cap V\in\wh{\mc{J}}$.
\end{obs}

\begin{exa}
Let $\mc{I}$ be a Borel ideal. The identity maps show that  $(\mc{I},\neg R_\mrm{ii},[\om]^\om)\leq_\mrm{BT}(\mc{I},\neg R_\mrm{ii},\mc{I}^+)$ holds. Conversely, if $\mc{I}$ is fragile with $Y=\om$ is the definition, then a trivial modification of the proof of Proposition \ref{edfin} shows that $(\mc{I},\neg R_\mrm{ii},\mc{I}^+)\leq_\mrm{BT}(\mc{I},\neg R_\mrm{ii},[\om]^\om)$ also holds, and so $\non^*(\mc{I},+)=\non^*(\mc{I},\infty)$, $\cov^*(\mc{I},+)=\cov^*(\mc{I},\infty)$, if $\PP$ destroys $\mc{I}$ then it $+$-destroys $\mc{I}$, and if $\vd_\PP [\om]^\om\cap V\notin \wh{\mc{I}}$ then $\vd_\PP \mc{I}^+\cap V\notin\wh{\mc{I}}$.
\end{exa}

\smallskip
Concerning analytic P-ideals and their $*$-destructibility, we know (basically \cite[Lem. 3.1]{towers}) that $(\mc{I},\subseteq^*,\mc{I})\leq_\mrm{BT}(\,\!^\om\om,\in^*,\mrm{Slm})$ for every such ideal where $\mrm{Slm}=\prod_{n\in\om}[\om]^{\leq n}$ is the family of {\em slaloms} on $\om$  (equipped with the product topology where $[\om]^{\leq n}$ is discrete) and $f\in^* S$ iff $f(n)\in S(n)$ for almost every $n$; and (see \cite[Cor. 524H]{Fr}) that $(\,\!^\om\om,\in^*,\mrm{Slm})\equiv_{\mrm{BT}}(\mc{N}\subseteq,\mc{N})$. In particular, if $\mc{I}$ is tall, then
\[\tag{Ia,Ib} \add^*(\mc{I},\infty)\geq \add(\mc{N})\;\;\text{and}\;\; \cof^*(\mc{I},\infty)\leq\cof(\mc{N}),\]\[\tag{IIa} \text{if}\;\PP\;\text{has the Sacks-property, then}\;\mc{I}\cap V\;\text{is cofinal in}\;\mc{I}\cap V^\PP,\]
\[\tag{IIb}  \text{if}\;\vd_\PP\bigcup\mc{N}\cap V\in\mc{N}\footnote{In other words, the union of Borel null sets coded in $V$ is a null set in $V^\PP$.},\;\text{then $\PP$ $*$-destroys}\;\mc{I}.\]

The question whether these inequalities and implications are actually equalities and equivalences for every tall analytic P-ideal is still open, but there are some partial results (see also after Corollary \ref{94}). An lsc submeasure $\varphi$ is {\em summable-like} if there is an $\eps>0$ such that for every $\de>0$ we can pick a sequence $(F_k)$ of pairwise disjoint finite sets and an $m\in\om$ such that (i) $\varphi(F_k)<\de$ for every $k$ and (ii) $\varphi(\bigcup_{k\in H}F_k)\geq \eps$ for every $H\in [\om]^m$. We say that an analytic P-ideal $\mc{I}$ is summable-like if $\mc{I}=\mrm{Exh}(\varphi)$ for some summable-like submeasure $\varphi$ (which implies that if $\mc{I}=\mrm{Exh}(\psi)$ then $\psi$ is also summable-like, this follows from \cite[Rem. 3.3]{repr}). For example, summable ideals which are not trivial modifications of $\mrm{Fin}$ and $\mrm{tr}(\mc{N})$ (see \cite{repr}) are summable-like. Also, if $\mc{I}\clrest X$ is summable-like for some $X\in\mc{I}^+$, then $\mc{I}$ is summable-like too. The next result is a special case of \cite[Thm. 3.7(ii)]{SoTo}.

\begin{prop}
$(\,\!^\om\om,\in^*,\mrm{Slm})\leq_\mrm{BT}(\mc{I},\subseteq^*,\mc{I})$ for every summable-like ideal $\mc{I}$, and hence these relations are BT-equivalent.
\end{prop}
\begin{proof}
For a strictly increasing $\bar{m}=(m_n)_{n\in\om}\in\,\!^\om\om$ let $\mrm{Slm}(\bar{m})=\prod_{n\in\om}[\om]^{\leq m_n}$ be the family of {\em $\bar{m}$-slaloms}, and define the relation $(\,\!^\om\om,\in^*,\mrm{Slm}(\bar{m}))$ as in the case of $m_n=n$.
\begin{claim} {\em (see \cite{SQ})}
$(\,\!^\om\om,\in^*,\mrm{Slm}(\bar{m}))\equiv_{\mrm{BT}}(\,\!^\om\om,\in^*,\mrm{Slm})$
\end{claim}
\begin{proof}[Proof of the Claim]
$(\,\!^\om\om,\in^*,\mrm{Slm}(\bar{m}))\leq_\mrm{BT}(\,\!^\om\om,\in^*,\mrm{Slm})$ is trivial. Conversely, fix a bijection $j:\om\to\,\!^{<\om} \om$; for $f\in\,\!^\om\om$ let $f'\in\,\!^\om\om$, $f'(n)=j^{-1}(f\clrest m_{n+1})$; and for $S\in\mrm{Slm}(\bar{m})$ and $n\in[m_k,m_{k+1})$ let
\[S'(n)=\big\{j(\ell)(n):\ell\in S(k)\;\text{and}\; |j(\ell)|=m_{k+1}\big\},\]
and if $n\in [0,m_0)$ then let $S'(n)=\0$. Notice that if $n\in [m_k,m_{k+1})$ then $|S(k)|\le m_k$ so $|S'(n)|\le m_k\le n$ for each $n$, hence $S\in\mrm{Slm}$.

Now if $f'\in^*S$, $f'(k)=j^{-1}(f\clrest m_{k+1})\in S(k)$ for $k\geq K$, then for every such $k$ and $n\in [m_k,m_{k+1})$ we have $f(n)=(f\clrest m_{k+1})(n)=j(f'(k))(n)\in S'(n)$, and hence $f\in^* S'$.
\end{proof}

Let $\mc{I}=\mrm{Exh}(\varphi)$ and $\eps>0$ from the definition of summable-likeness. For every $n$ pick a sequence $(F^n_k)_{k\in\om}$ of pairwise disjoint finite sets and an $m_n\in\om$ to $\de=2^{-n}$, that is, $\varphi(F^n_k)<2^{-n}$ for every $n,k$, and $\varphi(\bigcup\{F^n_k:k\in H\})\geq\eps$ for every $H\in [\om]^{m_n}$. We can assume that $\bar{m}=(m_n)_{n\in\om}$ is strictly increasing and, by shrinking the sequences $(F^n_k)_{k\in\om}$, that $F^n_k\cap F^{n'}_{k'}=\0$ whenever $n\ne n'$ or $k\ne k'$.

We will define a reduction $(\,\!^\om\om,\in^*,\mrm{Slm}(\bar{m}))\leq_\mrm{BT}(\mc{I},\subseteq^*,\mc{I})$. For $g\in\,\!^\om\om$ let $A_g=\bigcup\{F^n_{g(n)}:n\in\om\}$, then $A_g\in\mc{I}$ because lsc submeasures are $\sigma$-subadditive; and for $B\in\mc{I}$ let $S_B(n)=\{k:F^n_k\subseteq B\}$ if this set is of size $<m_n$, and $S_B(n)=\0$ otherwise. Notice that for almost every $n$ (say for $n\geq N_B$), we defined $S_B(n)$ according to the first option because otherwise there were $X\in [\om]^\om$ and $H_n\in [\om]^{m_n}$ ($n\in X$) such that $F^n_k\subseteq B$ for every $n\in X$ and $k\in H_n$, and hence $B$ contains infinitely many pairwise disjoint sets $\bigcup\{F^n_k:k\in H_n\}$ ($n\in X$) each of whom is of measure $\geq\eps$, and hence $B\notin \mc{I}$, a contradiction.

Now if $A_g\subseteq^* B$, then if $n$ is large enough, say $n\geq M$ then $\forall$ $k$ $(F^n_k\subseteq A_g \to F^n_k\subseteq B)$. In particular, if $n\geq M,N_B$ then $F^n_{g(n)}\subseteq B$, and hence $g(n)\in S_B(n)$, so $g\in^*S_B$.
\end{proof}

\begin{cor}\label{94}
If $\mc{I}$ is tall and summable-like then in (Ia,Ib) the inequalities are actually equalities, and in (IIa,IIb) the implications are actually equivalences.
\end{cor}

Concerning $*$-destructibility and combinatorics of analytic P-ideals, one of the most fundamental questions is if the above proposition, or at least the reverse inequalities in (Ia,Ib) and reverse implications in (IIa,IIb) hold for every tall analytic P-ideal.

Concerning non summable-like ideals, e.g. density ideals, in \cite{cardinvanalp} (applying results due to Fremlin and Farah), the authors proved that  $\add^*(\mc{Z}_{\vec{\mu}},\infty)=\add(\mc{N})$ and $\cof^*(\mc{Z}_{\vec{\mu}},\infty)=\cof(\mc{N})$ hold for every tall density ideal $\mc{Z}_{\vec{\mu}}$. Their proof is of purely combinatorial nature, is not via Borel Tukey connections.

Concerning $\mc{Z}$, there are strong indications that $(\,\!^\om\om,\in^*,\mrm{Slm})\leq_\mrm{BT}(\mc{Z},\subseteq^*,\mc{Z})$ may not hold (see e.g. \cite[Cor. 524H]{Fr} and \cite[Thm. 7]{LV}). At the same time, we already know that $\add^*(\mc{Z},\infty)=\add(\mc{N})$ and $\cof^*(\mc{Z},\infty)=\cof(\mc{N})$, and also (see \cite[Thm. 6.16]{FaSo}, based on Fremlin's proof of these last equalities) that the ``reverse'' (IIa) holds for $\mc{Z}$, that is, if $\mc{Z}\cap V$ is cofinal in $\mc{Z}\cap V^\PP$, then $\PP$ has the Sacks-property. The last missing implication, (IIb) for $\mc{Z}$ is still an open problem:

\begin{prob}
Does there exist a $\PP$ which $*$-destroys $\mc{Z}$ but does not add a slalom capturing all ground model reals?
\end{prob}

\section{Further questions}\label{secque}

Additional to the problems from the previous sections, here we list a couple of further questions we found interesting.

\subsection*{Destruction without collateral damage}
Fix two Borel ideals $\mc{I}$ and $\mc{J}$ and assume that there is no ``obvious'' reason why $\infty/+$-destruction of $\mc{J}$ would imply $\infty/+$-destruction of $\mc{I}$, e.g. $\mc{I}\nleq_\mrm{K}\mc{J}\clrest X$ for any $X\in\mc{J}^+$. One may ask if we can find a forcing notion $\PP$ which $\infty/+$-destroys $\mc{J}$ but does not $\infty/+$-destroy $\mc{I}$, or even $+$-destroys $\mc{J}$ without destroying $\mc{I}$, etc.

\subsection*{Destruction without adding unbounded reals}
First of all, by applying results from \cite{Lafl0}, we show that every $F_\sigma$ ideal can be $+$-destroyed by an $\,\!^\om\om$-bounding proper forcing notion. In \cite{Lafl0}, Laflamme showed that for every $F_\sigma$ ideal $\mc{I}$ there is an $\,\!^\om\om$-bounding proper forcing notion which destroys $\mc{I}$. This construction depends on certain parameters rather than directly on the ideal. We show that these parameters can be chosen such that the associated forcing notion $+$-destroys the ideal.

Let $\mc{I}=\bigcup_{n\in\om}\mc{C}_n$ be an $F_\sigma$ ideal on $\om$ ($\mc{C}_n$ compact) and $\mc{C}=\bigcup\{C\setminus n:C\in\mc{C}_n\}$. Then $\mc{C}$ is compact and $\mc{I}=\{F\cup C:F\in [\om]^{<\om},C\in\mc{C}\}$. Using this family, we can construct (see \cite[Lem. 3.2]{Lafl0} and below) a partition $(Q_n)_{n\in\om}$ of $\om$ into nonempty finite sets and families $\mc{E}_n\subseteq\mc{P}(Q_n)$ such that
\begin{itemize}
\item[(a)] $|\mc{E}_n|\geq n+1$ and $\bigcap\mc{E}\ne\0$ for every $\mc{E}\in [\mc{E}_n]^{n+1}$;
\item[(b)] $\forall$ $A\in\mc{I}$ $\exists$ $E\in\prod_{n\in\om}\mc{E}_n$ $|A\cap \bigcup_{n\in\om}E(n)|<\om$.
\end{itemize}
The forcing notion depends on these parameters (that is, on $(Q_n)$ and $(\mc{E}_n)$): $T\in \PP$ iff $T\subseteq \bigcup_{k\in\om}\prod_{n<k}[\mc{E}_n]^{n+1}$ is a perfect tree (where $\prod\0=\{\0\}$) such that
\[ \forall\;n\;\exists\;M\;\forall\;m\geq M\;\forall\;s\in T_m\;\forall\;\mc{E}\in [\mc{E}_m]^{n+1}\;|\{\mc{H}\in\mrm{ext}_T(s):\mc{E}\subseteq\mc{H}\}|\geq n+1\]
where $T_m=\{s\in T:|s|=m\}$ stands for the $m$th level of $T$. Then $\PP$ destroys $\mc{I}$ because if $e\in \prod_{n\in\om}[\mc{E}_n]^{n+1}$ is the generic real, then $\bigcup_{n\in\om}\bigcap e(n)\in [\om]^\om$ has finite intersection with every $A\in\mc{I}\cap V$. The nontrivial part (see \cite[Thm. 3.1]{Lafl0}) is that $\PP$ is proper and $\,\!^\om\om$-bounding.

\smallskip
Observe that the sequences $Q'_n:=Q_{n+1}$ and $\mc{E}'_n:=\mc{E}_{n+1}$ satisfy (a) and (b) above (on $\om\setminus Q_0$ with $\mc{I}\clrest (\om\setminus Q_0)$), and hence give rise to an analogous poset $\PP'$. If $e'$ is a $\PP'$-generic over $V$, then $E\cap \bigcap e'(n)\neq\emptyset$ for every $E\in \mc{E}'_n$ because $e'(n)\cup\{E\} \in [\mc{E}'_{n}]^{\leq n+2}= [\mc{E}_{n+1}]^{\leq n+2}$. Applying
(b) in $V^{\PP'}$ (it is a $\Ubf{\Pi}^1_2$ property) we conclude that $\bigcup_{n\in\om}\bigcap e'(n) \in\mc{I}^+$.

\smallskip
In particular, if $\mc{I}$ is an $F_\sigma$ ideal, then we can force $\mf{b}<\cov^*(\mc{I},+)$.

\begin{prob}
Which Borel ideals can be $+$-destroyed by $\,\!^\om\om$-bounding proper forcing notions?
\end{prob}

\subsection*{More general degrees of destruction}
We can define an even more general notion of destroying ideals as follows: Fix a Borel ideal $\mc{I}$ and a Borel  $\mc{D}\subseteq \bigcup\wh{\mc{I}}$. We say that $\PP$ {\em can $\mc{D}$-destroy} $\mc{I}$ if there is a $p\in\PP$ such that $p\vd\exists$ $D\in \mc{D}$ $\forall$ $A\in\mc{I}^V$ $|D\cap A|<\om$, and of course, we can define the cardinal invariants $\mrm{inv}^*(\mc{I},\mc{D})$ as well.  This notion raises a plethora of questions, for example: We know that $\mc{I}_{1/n}\subseteq\mc{Z}$, and hence $\mc{I}^+_{1/n}\supseteq\mc{Z}^+$. In particular, $\mbb{M}(\mc{Z}^*)$ $\mc{I}^+_{1/n}$-destroys $\mc{Z}$. Does there exist a forcing notion $\PP$ which $\mc{I}^+_{1/n}$-destroys $\mc{Z}$ but cannot $\mc{Z}^+$-destroy $\mc{Z}$?

\subsection*{Forcing with $\PP_{\wh{\mc{I}}}$} Let $\mc{I}$ be a tall analytic P-ideal. What can we say about the forcing notion $\PP_{\wh{\mc{I}}}=\mc{B}([\om]^\om)\setminus \wh{\mc{I}}$? We know that it is proper (see \cite[Section 4.6]{zap}) and it clearly destroys $\mc{I}$. Also, notice that $\mc{I}^+\in\PP_{\wh{\mc{I}}}$ is a condition and forces that $\mc{I}$ is $+$-destroyed, similarly, $\mc{I}^*$ forces that $\mc{I}$ is $*$-destroyed. In other words, it seems reasonable to decompose $\PP_{\wh{\mc{I}}}$ into three forcing notions $\PP(\mc{I},\infty)=\PP_{\wh{\mc{I}}\upharpoonright (\mc{I}\cap [\om]^\om)}$, $\PP(\mc{I},+)=\PP_{\wh{\mc{I}}\upharpoonright (\mc{I}^+\setminus\mc{I}^*)}$, and $\PP(\mc{I},*)=\PP_{\wh{\mc{I}}\upharpoonright \mc{I}^*}$.

\begin{prob}
Can $\PP(\mc{I},\infty)$ $+$-destroy or $\PP(\mc{I},+)$ $*$-destroy $\mc{I}$? If not, what can we say about their countable support  iterations? Do they add dominating etc reals?
\end{prob}

If $\mc{I}$ is not a $P$-ideal but $\cov^*(\mc{I},\infty)>\om$, then we can talk about the $\sigma$-ideal $\wh{\mc{I}}^\sigma\not\ni[\om]^\om$ generated by $\wh{\mc{I}}$ and the forcing notion $\PP(\mc{I},\infty)=\PP_{\wh{\mc{I}}^\sigma\upharpoonright (\mc{I}\cap [\om]^\om)}$. Similarly, if even $\cov^*(\mc{I},+)>\om$, then we can talk about $\PP(\mc{I},+)=\PP_{\wh{\mc{I}}^\sigma\upharpoonright (\mc{I}^+\setminus \mc{I}^*)}$ too. In particular, one can ask the  questions above about these forcing notions as well.

\end{document}